\pdfoutput=1
\documentclass[a4paper, 11pt, reqno]{amsart}

\usepackage{amssymb}
\usepackage{amsmath}
\usepackage{amsthm}
\usepackage[numbers]{natbib}
\usepackage{mathtools}
\usepackage{enumitem}
\usepackage{calc}
\usepackage{setspace}

\newcounter{dummy}
\numberwithin{dummy}{section}
\newtheorem*{main}{Main Theorem}

\newtheorem{thm}[dummy]{Theorem}
\newtheorem{defn}{Definition}[section]
\newtheorem{lem}[dummy]{Lemma}
\newtheorem{cor}[dummy]{Corollary}

\renewcommand\Im{\operatorname{Im}}

\makeatletter
\renewcommand{\@biblabel}[1]{[#1]\hfill}
\makeatother

\begin{document}
\title{Effective Approximation and Diophantine Applications}

\author{Gabriel. A. Dill}
\address{Departement Mathematik und Informatik, Spiegelgasse 1, CH-4051 Basel}
\email{gabriel.dill@unibas.ch}

\date{}

\begin{abstract}
Using the Thue-Siegel method, we obtain effective improvements on Liouville's irrationality measure for certain one-parameter families of algebraic numbers, defined by equations of the type $(t-a)Q(t)+P(t)=0$. We apply these to some corresponding Diophantine equations. We obtain bounds for the size of solutions, which depend polynomially on $a$, and bounds for the number of these solutions, which are independent of $a$ and in some cases even independent of the degree of the equation.
\end{abstract}

\subjclass[2010]{Primary 11D41; Secondary 11D45, 11D57, 11J68}

\keywords{Diophantine approximation, effective irrationality measures, Thue-Siegel method, polynomial bounds for Thue equations, counting solutions of Diophantine equations}

\maketitle

\thispagestyle{empty}

\section{Introduction}\label{sec:Introduction}
Let $\theta$ be an irrational algebraic number of degree $d$, then it is well-known that there is an effectively computable constant $C$ such that
\[\left|\theta-\frac{p}{q}\right| \geq \frac{1}{Cq^d}\]
for all integers $p$ and $q \geq 1$.

For $dÊ\geq 3$ we can do much better and Thue proved in his seminal paper \cite{T09} that any $\kappa > \frac{d}{2}+1$ can be taken as the exponent of $q$ in the denominator instead of $d$ to yield a bound of the form
\begin{equation}\label{eq:essentiallowerbound}
\left|\theta-\frac{p}{q}\right| \geq \frac{1}{Cq^\kappa}.
\end{equation}
This was improved to any
\[\kappa > \min_{e=1,\hdots,d-1}\left\{\frac{d}{e+1}+e\right\}\] 
by Siegel in his dissertation \cite{S21} (in particular, $\kappa = 2\sqrt{d}$ can be taken), then to $\kappa = \sqrt{2d}$ by Dyson in \cite{D47} and finally to any $\kappa > 2$ by Roth in \cite{R55}, which is almost best possible. Unfortunately, in all these improvements the constant $C$ is not effective.

The first effective improvement of the exponent to $\kappa=d-\epsilon$, $\epsilon=\epsilon(\theta)>0$, was found by Fel$'$dman and based on Baker's method of linear forms in logarithms; see \cite{F71} and \cite{BG96}, Corollary 1, for a later improvement. For special $\theta$, it is well-known that one can obtain better effective results, usually with the help of hypergeometric functions.

Using the Thue-Siegel method, Bombieri considered the equation
\begin{equation}\label{eq:bombierisequation}
t^{d}-at^{d-1}+1=0
\end{equation}
and proved that the unique real root $\theta > 1$ of this equation satisfies
\begin{equation}\label{eq:lowerboundbombieri}
\left| \theta - \frac{p}{q} \right| \geq \frac{1}{C(\theta)q^{39.2574}}
\end{equation}
for all integers $p$ and $qÊ\geq 1$ under the condition that $a \geq a_0(d)$ and $q \geq q_0(d)$ for effectively computable $C(\theta)$, $a_0(d)$ and $q_0(d)$ (cf. \cite{B82}, p. 294, Example 3).

The purpose of the present paper is threefold:

First, we calculate explicit values of $C(\theta)$ and $a_0$ in results like \eqref{eq:lowerboundbombieri} with particular attention to the parameter $a$, getting rid of $q_0$ in the process. For example, if $d \geq 23$ and $\left|a\right| \geq 2^{1196d}$, we get exponent $22.99$ and
\begin{equation}
C(\theta) = C(d,a) = 2^{84d^2}\left|a\right|^{28d^2}
\end{equation}
for the unique real root $\theta=\xi$ of \eqref{eq:bombierisequation} with $\left|\xi-a\right|<1$, as we will see in Corollary \ref{cor:bombieripoly}. By the way, this $\xi$ indeed has degree exactly $d$ --- see Lemma \ref{lem:rouche} for this and more.

Second, we apply our approximation results to Diophantine equations. That corresponding to \eqref{eq:bombierisequation} is
\begin{equation}\label{eq:bombierisdiophantineequation}
x^d-ax^{d-1}y+y^{d}=m
\end{equation}
with $m \in \mathbb{Z}$. But results like \eqref{eq:lowerboundbombieri} as they stand do not imply anything for the (integral) solutions of \eqref{eq:bombierisdiophantineequation}, due to the other roots of \eqref{eq:bombierisequation}. It turns out that with minor conditions one can render these non-real and therefore harmless, thus leading to bounds provided that $\left|a\right|$ is large enough.

But even this proviso can be eliminated with linear forms in logarithms and we obtain as an immediate consequence of Theorem \ref{cor:bombieriequation} that any solution of
\[Êx^d-ax^{d-1}y+y^d=m\]
in integers $x$ and $y$, where $d \geq 23$, $d$ odd and $a \leq -4$, satisfies
\[\max\{\left|x\right|,\left|y\right|\} \leq c(d)\left|a\right|^{\lambda(d)}\left|m\right|^{\mu(d)}\]
for some explicit functions $c$, $\lambda$ and $\mu$ of $d$.

Here, the polynomial dependence on $m$ is a familiar feature; however the exponent usually depends on the coefficients of the left-hand side of \eqref{eq:bombierisdiophantineequation}. Our exponent depends only on $d$. The polynomial dependence on $a$, which comes directly from our approximation results, is much less common in the literature.

Since $x=a$, $y=m=1$ is a solution, we see that $\max\{\left|x\right|,\left|y\right|\}$ grows polynomially in $a$. Similarly taking $x^d=m$, $y=0$ we see that it grows polynomially in $\left|m\right|$. In this sense, our result is best possible even though the exponents $\lambda(d)$ and $\mu(d)$ probably are not.

It is also in line with a folklore belief that all integer solutions of
\begin{equation}\label{eq:essentialdiophantineineq}
F(x,y)= m
\end{equation}
can be bounded polynomially in terms of $m$ and the coefficients of $F$, where $F$ is an irreducible binary form of degree $dÊ\geq 3$ with integer coefficients.

Thirdly, we combine our results with gap arguments to find upper bounds for the number of solutions. It turns out that our approximation results are usually strong enough to yield bounds which are independent of $a$. This feature occurs already in the well-known result of Bombieri and Schmidt that \eqref{eq:essentialdiophantineineq} has at most $215d$ solutions in the case $m=1$ if $d$ is large enough (see \cite{BS87}, pp. 69sq.).

However, our bounds are in some cases even independent of $d$; for example, we will show in Section \ref{sec:ApplicationsII} that the number of integral solutions of 
\[(x-ay)(x^2+y^2)^{\frac{d-1}{2}}-y^d=x+y\]
is at most $11$, independently of both $d$ and $a$, if $d \geq 25$ is odd and (unfortunately) $\left|a\right| \geq 2^{164d}$. This particular equation is of no special significance, but has been chosen as an example to illustrate the method.

The independence of $d$ also occurs in work of Mueller and Schmidt on $F(x,y)=f(x,y)$ provided that the number of non-zero coefficients of $F$ and the degree of $f$ are bounded independently of $d$ (see \cite{MS95}, pp. 332sq.). However, our $F(x,y)$ has the maximum number $d+1$ of non-zero coefficients.

Our proofs of \eqref{eq:lowerboundbombieri} use the Thue-Siegel method like those of Bombieri in \cite{B82} and \cite{B87}, except that we replace his use of the Dyson lemma with a much simpler zero estimate. This idea as well as the basis of Section \ref{sec:Main Theorem} and the inspiration for this whole article is owed to an unpublished work by Masser --- an outline appeared in \cite{BC03}, Appendix, pp.\ 59--61. On the whole, we refrain from using some of the more intricate ideas and estimates in Bombieri's work; in return we are able to easily calculate explicit values for all occurring constants.

\section{Preliminaries}\label{sec:Preliminaries}

If $\tau$ is any real number, we denote by $\left[\tau\right]$ the largest integer which is smaller than or equal to $\tau$. For a polynomial $P$ in two variables $x$ and $y$ and an integer $l \geq 0$, we write
\[ÊP_{l}(x,y)=\frac{1}{l!}\frac{\partial^l P}{\partial x^l}\left(x,y\right).\]

We refer to \cite{BG06}, section 1.5, for the definition of the absolute multiplicative height $H(\underline{\alpha})$ of $\underline{\alpha} \in \mathbb{P}^N(K)$, where $K$ is a number field, as well as the definition of the absolute multiplicative height or simply the height $H(\alpha)$ of $\alpha \in K$. We will also need the height of a linear form.
\begin{defn}\label{defn:linearheight}
The projective absolute height $H(L)$ of a linear form
\[L(x_1,\hdots,x_N)=\alpha_1 x_1+\cdots+\alpha_N x_N\]
with coefficients in $K$ is defined by
\[H(L)=H(\alpha_1:\cdots:\alpha_N).\]
\end{defn}

One finds that
\begin{equation}\label{eq:heightone}
H\left(\frac{p}{q}\right)=\max\{\left|p\right|,\left|q\right|\}
\end{equation}
for coprime integers $p$ and $qÊ\neq 0$. We further recall that for algebraic $\alpha$ and $\beta$ we have
\begin{equation}\label{eq:heighttwo}
H(\alpha+\beta) \leq 2H(\alpha)H(\beta)
\end{equation}
and
\begin{equation}\label{eq:heightthree}
\left|\alpha\right| \geq H(\alpha)^{-d} \quad (\alpha \neq 0).
\end{equation}
Using these properties, one easily proves
\begin{equation}\label{eq:imaginarypartbound}
\left|\Im\theta\right|=\frac{1}{2}\left|\theta-\overline{\theta}\right|\geq \frac{1}{2}\left(2H(\theta)^2\right)^{-d^2},
\end{equation}
if $\theta$ is non-real and $\overline{\theta}$ is the complex conjugate of $\theta$, since $[\mathbb{Q}(\theta,\overline{\theta}):\mathbb{Q}] \leq d^2$. This bound will be used in Section \ref{sec:Applications} to deal with the non-real zeroes of $F$.

\begin{lem}\label{lem:Siegel}
Let $K$ be an algebraic number field of degree $d$ and let $M$, $N$ be positive integers with $N > dM$. Put $\mu=\frac{dM}{N-dM}$. Let $L_1(x_1,\hdots,x_N)$, \dots, $L_M(x_1,\hdots,x_N)$ be linear forms with coefficients in $K$ and projective absolute heights at most $\mathcal{H} \geq 1$. Then there exist rational integers $x_1$, \dots, $x_N$, not all zero, of absolute values at most $(\sqrt{N}\mathcal{H})^{\mu}$, such that
\[ L_1(x_1,\hdots,x_N)=\cdots=L_M(x_1,\hdots,x_N)=0.\]
\end{lem}

\begin{proof}
See \cite{BG06}, p.\ 79, Theorem 2.9.19, with $K=\mathbb{Q}$, $F=K$, $r=d$ and $A_i$ given by the coefficients of $L_i$ ($i=1,\hdots,M$). Note that
\[ÊH_{Ar}(A_i) \leq \sqrt{N}H(L_i) \leq \sqrt{N}\mathcal{H} \quad (i=1,\hdots,M).\]
We deduce that at least one non-zero vector $\underline{x}=(x_1,\hdots,x_N) \in \mathbb{Z}^{N}$ exists with $L_1(x_1,\hdots,x_N)=\cdots=L_M(x_1,\hdots,x_N)=0$ and $H(\underline{x}) \leq (\sqrt{N}\mathcal{H})^{\mu}$, where $H(\underline{x})$ is defined by considering $\underline{x}$ as an element of $\mathbb{P}^{N-1}(\mathbb{Q})$. After dividing out any common factors, we can assume without loss of generality that $\gcd(x_1, \hdots, x_N)=1$. Since this implies that
\[H(\underline{x})=\max\{\left|x_1\right|,\hdots,\left|x_N\right|\},\]
the lemma follows.
\end{proof}

\begin{defn}\label{defn:polyheight}
Let $P(x_1,\hdots,x_n)=\sum_{I=(i_1,\hdots,i_n)}{a_Ix_1^{i_1}\cdots x_n^{i_n}}$ be a polynomial in $\mathbb{C}[x_1,\hdots,x_n]$. The size $\left|P\right|$ and the length $L(P)$ of $P$ are defined by
\[Ê\left|P\right|=\max_{I} \left|a_I\right|, \quad L(P)=\sum_{I}Ê\left|a_I\right|. \]
\end{defn}

\begin{lem}\label{lem:Gelfond}
Suppose that $P_1(x)$ and $P_2(x)$ are polynomials in $\mathbb{C}[x]$ of degrees $n_1$ and $n_2$ respectively. Then
\[Ê\left|P_1\right|\left|P_2\right| \leq 2^n\left|P_1P_2\right|,\]
where $n=n_1+n_2$.
\end{lem}

\begin{proof}
See \cite{BG06}, p.\ 27, Lemma 1.6.11 with $m=2$, $f_1=P_1$, $f_2=P_2$ and $d=n$.
\end{proof}

\begin{defn}\label{defn:wronskian}
Let $P_1$, \dots, $P_n$ be polynomials in $\mathbb{C}[x]$. The Wronskian $W$ of $P_1$, \dots, $P_n$ is defined by
\[ÊW(x)=
\begin{vmatrix}
P_1(x) & \hdots & P_n(x) \\
P_1'(x) & \hdots & P_n'(x) \\
\vdots & \ddots & \vdots \\
P_1^{(n-1)}(x) &Ê\hdots & P_n^{(n-1)}(x)
\end{vmatrix}.\]
\end{defn}

We will need the following well-known lemma about Wronskians.

\begin{lem}\label{lem:Wronski}
Let $P_1$, \dots, $P_n$ be linearly independent polynomials in $\mathbb{C}[x]$ and $W$ their Wronskian. Then $W \neq 0$.
\end{lem}

\begin{lem}\label{lem:Binom}
Let $N \geq 1$, $0 \leq n \leq N$ and $l \geq 0$ be integers. Then
\[Ê\binom{n}{l} \leq \sqrt{\frac{2}{\pi}}\frac{2^N}{\sqrt{N}}.\]
\end{lem}

We use the convention that $\binom{n}{l}=0$ for $l > n$.

\begin{proof}
If $n=2m$ is even, then ${n \choose l}\leq {2m \choose m}$ and one checks that 
$$2^{-4m}(2m){2m \choose m}^2={1\over 2}\prod_{r=2}^m\left(1+{1 \over 4r(r-1)}\right)<{1 \over 2}\prod_{r=2}^\infty\left(1+{1 \over 4r(r-1)}\right)={2 \over \pi}$$
so that ${2m \choose m}<\sqrt{2 \over \pi}{2^{2m}\over \sqrt{2m}}\leq \sqrt{2 \over \pi}{2^{N}\over \sqrt{N}}$ and the lemma follows. And if $n=2m+1$ is odd, then 
$${n \choose l}\leq {2m+1 \choose m+1}={1 \over 2}{2m+2 \choose m+1}<{1 \over 2}\sqrt{2 \over \pi}{2^{2m+2}\over \sqrt{2m+2}}<\sqrt{2 \over \pi}{2^{2m+1}\over \sqrt{2m+1}}$$
and the lemma follows here too.
\end{proof}

\section{Main Theorem}\label{sec:Main Theorem}

\begin{main}
Suppose that $\theta$ is real algebraic of degree $d \geq 3$ and height $H \geq 1$. Fix an integer $e$ with
\[ 1 \leq e < d \]
and $\epsilon$ with
\[ 0 < \epsilon < 1.\]
Put
\[ \delta = \frac{d+\epsilon}{e+1}, \quad \alpha=\frac{d\delta}{\epsilon}\]
\[ \beta = d\delta+\alpha, \quad \gamma=1-\epsilon > 0.\]
Suppose the integers $p_0$ and $q_0 \geq 1$ satisfy
\[ \LambdaÊ= c^{-1}\left|\theta-\frac{p_0}{q_0}\right|^{-\gamma} > 1,\]
where 
\[c=q_0^{\delta}(2H)^{\beta}.\]
Then the effective strict type of $\theta$ is at most
\[ \kappa = e\left(1+\frac{\log c}{\log\Lambda}\right).\]
More precisely
\[Ê\left|\theta-\frac{p}{q}\right| \geq \frac{1}{Cq^{\kappa}}\]
for all integers $p$ and $q \geq 1$, where
\[ÊC=c\tilde{c}^{\frac{\kappa}{e}}\left|\theta-\frac{p_0}{q_0}\right|^{d(e-\kappa)}\]
and
\[ \tilde{c}=2^{e+\frac{\alpha}{2\delta}+2}(e+1)^{\frac{\alpha}{2\delta}+1}H^{\frac{eÊ\beta}{\delta}}.\]
\end{main}

For the proof, three lemmata are needed. We remark that $\left|\theta-\frac{p_0}{q_0}\right|$ can be bounded from below in terms of only $H$ and $q_0$ and hence $C$ essentially depends only on $d$, $H$ and $q_0$ as well as the choice of the parameters $e$ and $\epsilon$. The dependence on the unknown $q_0$, which might be arbitrarily large, if it exists at all, is the reason for the general ineffectivity of the Thue-Siegel method.

One can see now that Siegel's theorem directly follows from the main theorem as follows: If there were infinitely many approximations $\frac{r}{s}$ with $\left|\theta-\frac{r}{s}\right| < s^{-\lambda}$ for some $\lambda > e+\frac{d}{e+1}$ and $1Ê\leq eÊ\leq d-1$, we could obtain an asymptotic strict type of
\[Êe\left(1+\frac{\delta}{-\delta+\gamma\lambda}\right)\]
which tends to $e+\frac{ed}{(e+1)\lambda-d} < e+\frac{d}{e+1}$ for $sÊ\to \infty$ and $\epsilon \to 0$, a contradiction.

\begin{lem}\label{lem:poly}
For each natural number $k$ there exists a polynomial $P \neq 0$ in $\mathbb{Z}[x,y]$ of degree at most $\delta k$ in $x$ and at most $e$ in $y$ such that
\[ÊP_{l}(\theta,\theta)=0 \quad (0 \leq l < k)\]
and $\left|P\right| \leq c_1(2H)^{\alpha k}$ with
\[c_1=2^{e+\frac{\alpha}{2\delta}}(e+1)^{\frac{\alpha}{2\delta}}H^{\frac{\alpha e}{\delta}}.\]
Furthermore, the polynomial $P$ is not divisible by any non-constant polynomial in $\mathbb{C}[y]$.
\end{lem}

Apart from the non-divisibility clause, which will be crucial to reach effectivity, this is exactly the construction from Siegel's proof.

\begin{proof}
Write $D=\left[\delta k\right] \geq 1$ and set
\[ÊP(x,y)=\sum_{i=0}^{D}\sum_{j=0}^{e}{p_{ij}x^{i}y^{j}}.\]
We have to solve
\begin{equation}\label{eq:linear}
0=P_{l}(\theta,\theta)=\sum_{i=0}^{D}\sum_{j=0}^{e}{p_{ij}\binom{i}{l}\theta^{i-l+j}} \quad (0Ê\leq l < k)
\end{equation}
for the rational integers $p_{ij}$. These are $M=k$ homogeneous linear equations with coefficients in the field $K=\mathbb{Q}(\theta)$ of degree $d$ over $\mathbb{Q}$. The number of unknowns is 
\[N=(D+1)(e+1) > \delta k(e+1)=dk+\epsilon k > dM.\]

Therefore, we may apply Lemma \ref{lem:Siegel} with
\[Ê\mu = \frac{dk}{(D+1)(e+1)-dk} < \frac{dk}{\epsilon k}=\frac{\alpha}{\delta}\]
and get rational integers $p_{ij}$, not all zero, satisfying \eqref{eq:linear}, of absolute values at most
\begin{equation}\label{eq:siegelbound}
\{\sqrt{(D+1)(e+1)}\mathcal{H}\}^{\frac{\alpha}{\delta}},
\end{equation}
where $\mathcal{H} \geq 1$ is an upper bound for the projective absolute height of the linear forms in \eqref{eq:linear}.

Let $H_l$ be the projective absolute height of the $l$-th linear form ($0Ê\leq l < k$). By Definition \ref{defn:linearheight} we have
\[H_l^{d}=\prod_{v}{\left(\max_{i,j}\left|\binom{i}{l}\theta^{i-l+j}\right|_v\right)}.\]
We put
\[Ê\lambda_v = \max\left(1,\left|\theta\right|_v\right).\]

If $v$ is an infinite valuation, we can use Lemma \ref{lem:Binom} to get
\[\left|\binom{i}{l}\theta^{i-l+j}\right|_vÊ\leq \binom{i}{l}\lambda_v^{i-l+j} \leq \frac{2^D}{\sqrt{D}}\lambda_v^{D+e} \quad (0 \leq i \leq D, \mbox{ } 0Ê\leq j \leq e).\]
If $v$ is a finite valuation, then
\[\left|\binom{i}{l}\theta^{i-l+j}\right|_v \leq \lambda_v^{i-l+j}Ê\leq \lambda_v^{i+j} \leq \lambda_v^{D+e} \quad (0 \leq i \leq D, \mbox{ } 0Ê\leq j \leq e).\]

Since there are $d$ infinite valuations, we obtain
\[ÊH_l^d \leq D^{-\frac{d}{2}}2^{Dd}\left(\prod_{v}{\lambda_v}\right)^{D+e}=D^{-\frac{d}{2}}2^{Dd}(H^d)^{D+e}\]
by multiplying up all these estimates. Thus $H_l \leq \frac{2^{D}}{\sqrt{D}}H^{D+e}$, and we can take
\[Ê\mathcal{H} = \frac{2^{D}}{\sqrt{D}}H^{D+e}\]
in \eqref{eq:siegelbound}. This yields the estimate
\[Ê\left|P\right|=\max \left|p_{ij}\right| \leq \left\{\sqrt{\frac{(D+1)(e+1)}{D}}2^{D}H^{D+e}\right\}^{\frac{\alpha}{\delta}}.\]

Using $D+1Ê\leq 2D$, we deduce that
\[ \left|P\right| \leq (\sqrt{2(e+1)}2^{D}H^{D+e})^{\frac{\alpha}{\delta}} = c_2(2H)^{\frac{\alpha D}{\delta}}\]
with
\[c_2=2^{\frac{\alpha}{2\delta}}(e+1)^{\frac{\alpha}{2\delta}}H^{\frac{\alpha e}{\delta}}=2^{-e}c_1.\] 
Since $D \leq \delta k$, this implies that
\begin{equation}\label{eq:firstbound}
\left|P\right| \leq c_2(2H)^{\alpha k}.
\end{equation}

We now have a non-zero polynomial $P$ which satisfies the linear equations, but the non-divisibility clause is not necessarily fulfilled. To remedy this, we write
\[ P(x,y)=\sum_{i=0}^{D}{Q_i(y)x^{i}}\]
with $Q_0$, \dots, $Q_D$ in $\mathbb{Z}[y]$ not all zero. Let $Q$ be the greatest common divisor of $Q_0$, \dots, $Q_D$ in $\mathbb{Q}[y]$. We may take $Q$ in $\mathbb{Z}[y]$ and primitive. Put
\[Ê\tilde{Q_i}=\frac{Q_i}{Q} \quad (0 \leq iÊ\leq D).\]

By the Gauss lemma, the $\tilde{Q_i}$ are in $\mathbb{Z}[y]$ and they are coprime in $\mathbb{Q}[y]$ and by extension in $\mathbb{C}[y]$. We can estimate their sizes with Lemma \ref{lem:Gelfond}. Since the $Q_i$ have degrees at most $e$, we get
\[|\tilde{Q_i}| \leq |\tilde{Q_i}|\left|Q\right| \leq 2^{e}\left|Q_i\right| \quad (0 \leq i \leq D),Ê\]
using that $Q$ is a non-zero polynomial with integer coefficients and therefore its size is at least 1. Now it follows from \eqref{eq:firstbound} that
\[Ê|\tilde{Q_i}| \leq 2^{e}c_2(2H)^{\alpha k}= c_1(2H)^{\alpha k} \quad (0Ê\leq i \leq D).\]

We put
\[Ê\tilde{P}(x,y)=\frac{P(x,y)}{Q(y)}=\sum_{i=0}^{D}{\tilde{Q_i}(y)x^i}.\]
This polynomial is non-zero, has rational integer coefficients and degree at most $\delta k$ in $x$ and at most $e$ in $y$. By the above, it satisfies
\[Ê|\tilde{P}| \leq c_1(2H)^{\alpha k}\]
and it has no non-constant factor in $\mathbb{C}[y]$, because such a factor would have to divide all the $\tilde{Q_i}$ ($0 \leq i \leq D$). Furthermore, we have
\[ÊQ(\theta)\tilde{P}_l(\theta,\theta)=P_l(\theta,\theta)=0 \quad (0 \leq l < k).\]
Since $Q$ is a non-zero polynomial in $\mathbb{Z}[y]$ of degree at most $e < d$ and $\theta$ is of degree $d$, we must have $Q(\theta)Ê\neq 0$. It follows that
\[ \tilde{P}_l(\theta,\theta)=0 \quad (0 \leq lÊ< k),\]
so $\tilde{P}$ satisfies all conditions of the lemma.
\end{proof}

\begin{lem}\label{lem:zeroestimate}
Let $P$ be as in Lemma \ref{lem:poly}. Suppose $k \geq e$ and let $\xi$ and $\eta$ be arbitrary complex numbers with $\xi$ not a conjugate of $\theta$. Then there exists an integer $l$ with
\[Ê0 \leq l \leq \epsilon k + ed\]
and
\[ÊP_l(\xi,\eta) \neq 0.\]
\end{lem}
In Siegel's proof, $\eta$ here was required to have large height, growing exponentially in $k$, which hampered effectivity. The non-divisibility clause in Lemma \ref{lem:poly} allows $\eta$ to be chosen arbitrarily.
\begin{proof}
We write
\begin{equation}\label{eq:ysplit}
P(x,y)=\sum_{j=0}^{e}{P_j(x)y^j}.
\end{equation}
Since the $P_j$ are not all zero, their rank $f+1$ over $\mathbb{Q}$ satisfies $0 \leq fÊ\leq e$. Let $A_0$, \dots, $A_f$ be a subset of $P_0$, \dots, $P_e$ forming a basis for the vector space they generate over $\mathbb{Q}$. Thus $A_0$, \dots, $A_f$ are in $\mathbb{Z}[x]$ of degree at most $D=[\delta k]$. By writing $P_0$, \dots, $P_e$ as rational linear combinations of $A_0$, \dots, $A_f$ and substituting into \eqref{eq:ysplit}, we see that
\begin{equation}\label{eq:ysplitred}
P(x,y)=A_0(x)B_0(y)+\hdots+A_f(x)B_f(y)
\end{equation}
for polynomials $B_0$, \dots, $B_f$ in $\mathbb{Q}[y]$ of degree at most $e$. These polynomials are coprime, because any non-constant common factor in $\mathbb{C}[y]$ would have to divide $P$. In particular, they have no common zero. Let $W$ be the Wronskian of $A_0$, \dots, $A_f$. Since $A_0$, \dots, $A_f$ are linearly independent over $\mathbb{Q}$ (and hence over $\mathbb{C}$), we have $W \neq 0$ by Lemma \ref{lem:Wronski}.

On the other hand, we claim that $W(x)$ has a zero of high order at $x=\theta$. As $B_0$, \dots, $B_f$ have no common zero, we can assume without loss of generality that $B_0(\theta)Ê\neq 0$. By definition we have
\[B_0(y)W(x)=\begin{vmatrix}
A_0(x)B_0(y) & A_1(x) & \hdots & A_f(x) \\
A_0'(x)B_0(y) & A_1'(x) & \hdots & A_f'(x) \\
\vdots &Ê\vdots & \ddots & \vdots \\
A_0^{(f)}(x)B_0(y) &ÊA_1^{(f)}(x) & \hdots & A_f^{(f)}(x)
\end{vmatrix},\]
which yields
\begin{equation}\label{eq:bigwronski}
B_0(y)W(x)=\begin{vmatrix}
P(x,y) & A_1(x) & \hdots & A_f(x) \\
P_1(x,y) & A_1'(x) & \hdots & A_f'(x) \\
\vdots &Ê\vdots & \ddots & \vdots \\
f!P_f(x,y) &ÊA_1^{(f)}(x) & \hdots & A_f^{(f)}(x)
\end{vmatrix}
\end{equation}
by elementary column operations. Since $P(x,\theta)$ has a zero of order at least $k$ at $x=\theta$, the polynomial $P_l(x,\theta)$ has a zero of order at least $k-l$ at $x=\theta$ ($0Ê\leq l \leq k$). Note that $f \leq e \leq k$. Since $B_0(\theta)\neq 0$, we see by putting $y=\theta$ in \eqref{eq:bigwronski} that $W(x)$ has a zero of order at least $k-f$ at $x=\theta$. This is the above claim.

If $F$ is the minimal polynomial of $\theta$ in $\mathbb{Q}[x]$, it follows that
\begin{equation}\label{eq:wronskifactors}
W(x)=(F(x))^{k-f}R(x)
\end{equation}
with $R$ in $\mathbb{Q}[x]$, because $W$ is in $\mathbb{Q}[x]$.

Next we estimate the degree of $R$. We see from the definition that the degree of $W$ is at most
\[ÊD+(D-1)+\hdots+(D-f)=(f+1)D-\frac{1}{2}f(f+1),\]
which in turn is at most
\[Ê(e+1)D-f \leq (e+1)\delta k -f=dk+\epsilon k-f.\]
Since the degree of $F$ is $d$, it follows from \eqref{eq:wronskifactors} that the degree of $R$ is at most
\begin{equation}\label{eq:degreeboundrest}
dk+\epsilon k -f-d(k-f) \leq \epsilon k+ed-f.
\end{equation}

We now assume that the lemma is false, i.e.
\begin{equation}Ê\label{eq:zero}
P_l(\xi,\eta)=0
\end{equation}
for all integers $l$ with
\begin{equation} \label{eq:lrange}
0 \leq l \leq L=\epsilon k+ed.
\end{equation}
We will show that this forces $W=0$ and hence a contradiction. Using \eqref{eq:ysplitred}, we deduce from \eqref{eq:zero} that
\[A_0^{(l)}(\xi)B_0(\eta)+\hdots+A_f^{(l)}(\xi)B_f(\eta)=0\]
for all $l$ satisfying \eqref{eq:lrange}. Because the $B_i$ have no common zero, the $B_i(\eta)$ are not all zero ($0 \leq i \leq f$). It follows that
\begin{equation}Ê\label{eq:detzero}
\begin{vmatrix}
A_0^{(l_0)}(\xi) & \hdots & A_f^{(l_0)}(\xi) \\
\vdots & \ddots & \vdots \\
A_0^{(l_f)}(\xi) &Ê\hdots & A_f^{(l_f)}(\xi)
\end{vmatrix}=0
\end{equation}
for all integers $l=l_0$, \dots, $l_f$ satisfying \eqref{eq:lrange}.

Next let $t$ be an integer with $0 \leq tÊ\leq L-f$. For non-negative integers $t_0$, \dots, $t_f$ with $t_0+\hdots+t_f=t$ put
\[W_{t_0,\hdots,t_f}(x)=\begin{vmatrix}
A_0^{(t_0)}(x) & A_1^{(t_0)}(x) & \hdots & A_f^{(t_0)}(x) \\
A_0^{(t_1+1)}(x) & A_1^{(t_1+1)}(x) & \hdots & A_f^{(t_1+1)}(x) \\
\vdots & \vdots & \ddots & \vdots \\
A_0^{(t_f+f)}(x) & A_1^{(t_f+f)}(x) &Ê\hdots & A_f^{(t_f+f)}(x)
\end{vmatrix}.\]
Now the integers $l_i=t_i+i$ ($0 \leq i \leq f$) satisfy
\[Ê0 \leq l_i=t_i+i \leq t+f \leq L \quad (0 \leq i \leq f)\]
and hence we conclude with \eqref{eq:detzero} that
\[ÊW_{t_0,\hdots,t_f}(\xi)=0.\]

By applying the generalized product rule to every summand in the determinant $W(x)$ and regrouping the summands afterwards, we see that
\[ÊW^{(t)}(x)=\sum_{\substack{t_0,\hdots, t_f \geq 0 \\ t_0+\hdots+t_f=t}}{\binom{t}{t_0, t_1, \hdots, t_f}W_{t_0,\hdots,t_f}(x)}\]
and therefore
\[ÊW^{(t)}(\xi)=0 \quad (0Ê\leq t \leq L-f).\]

Thus $W(x)$ has a zero of order at least $[L-f]+1$ at $x=\xi$. Since $\xi$ is not a conjugate of $\theta$, we have $F(\xi)Ê\neq 0$. Thus, by \eqref{eq:wronskifactors}, $R(x)$ also has a zero of order at least $[L-f]+1$ at $x=\xi$. But by \eqref{eq:degreeboundrest} and \eqref{eq:lrange} its degree is at most $L-f$. It follows that $R=0$ and therefore $W=0$, which is the desired contradiction.
\end{proof}

\begin{lem}\label{lem:upperestimate}
Suppose $p_0$, $q_0$, $p$ and $q$ are integers with $q_0 \geq 1$, $q \geq 1$ and
\[Ê\left|\theta-\frac{p_0}{q_0}\right| < 1, \quad \left|\theta-\frac{p}{q}\right| < 1.\]
Then for any $k \geq \frac{ed}{\gamma}$, we have
\[Êq_0^{-\delta k}q^{-e} \leq c_3(2H)^{\beta k}\left\{\left|\theta-\frac{p_0}{q_0}\right|^{\gamma k - ed}+\left|\theta-\frac{p}{q}\right|\right\}\]
with $c_3 = 2(e+1)H^{ed}c_1=2^{e+\frac{\alpha}{2\delta}+1}(e+1)^{\frac{\alpha}{2\delta}+1}H^{\frac{e\beta}{\delta}}$.
\end{lem}
\begin{proof}
Since $\gamma = 1-\epsilon < 1$, we have $k \geq \frac{ed}{\gamma} \geq e$, and therefore we can apply Lemma \ref{lem:zeroestimate} with $\xi = \frac{p_0}{q_0}$ and $\eta =Ê\frac{p}{q}$. Note that $\xi$ is not a conjugate of $\theta$, because $\theta$ is of degree $d \geq 3$ over $\mathbb{Q}$. We get an integer $l$ with \eqref{eq:lrange} such that
\[Ê\lambda = P_l\left(\frac{p_0}{q_0},\frac{p}{q}\right) \neq 0.\]
Since 
\[P_l(x,y)=\sum_{i=0}^{D}\sum_{j=0}^{e}{p_{ij}\binom{i}{l}x^{i-l}y^{j}}\]
is a polynomial with integer coefficients of degree at most $D \leq \delta k$ in $x$ and at most $e$ in $y$, it follows that
\begin{equation}\label{eq:lambdalowerbound}
\left|\lambda\right| \geq q_0^{-\delta k}q^{-e}.
\end{equation}

On the other hand, the mean value theorem implies that
\begin{equation}\label{eq:meanvalueone}
P_l\left(\frac{p_0}{q_0},\frac{p}{q}\right)-P_l\left(\frac{p_0}{q_0},\theta\right) = \frac{\partial P_l}{\partial y}\left(\frac{p_0}{q_0},\theta'\right) \left(\frac{p}{q}-\theta\right)
\end{equation}
for some $\theta'$ between $\theta$ and $\frac{p}{q}$. As $\left|\theta-\frac{p}{q}\right|<1$, it follows that $\left|\theta'-\theta\right| < 1$.

Next we put $f(x)=P_l(x,\theta)$. Now \eqref{eq:lrange} says
\begin{equation}\label{eq:newlbound}
l \leq (1-\gamma)k+ed = k -Ê(\gamma k -ed).
\end{equation}
Since $k \geq \frac{ed}{\gamma}$, we have $l \leq k$ and so $f(x)$ has a zero of order at least $k-l$ at $x=\theta$. It follows from Taylor's theorem with remainder that
\[Êf\left(\frac{p_0}{q_0}\right) = \frac{f^{(k-l)}(\theta'')}{\left(k-l\right)!}\left(\frac{p_0}{q_0}-\theta\right)^{k-l} \]
for some $\theta''$ between $\theta$ and $\frac{p_0}{q_0}$. As before, this implies that $\left| \theta''-\thetaÊ\right| < 1$.

Since $f(x)=P_l(x,\theta)$, it follows that
\[ P_l\left(\frac{p_0}{q_0},\theta\right) = \frac{1}{l!\left(k-l\right)!}\frac{\partial^k P}{\partial x^k}(\theta'',\theta)\left(\frac{p_0}{q_0}-\theta\right)^{k-l}\]
and hence
\[P_l\left(\frac{p_0}{q_0},\theta\right) =\binom{k}{l}P_k(\theta'',\theta)\left(\frac{p_0}{q_0}-\theta\right)^{k-l}.\]
Combining this with \eqref{eq:meanvalueone}, we deduce that
\begin{equation}\label{eq:meanvaluefull}
\lambda = \binom{k}{l}P_k(\theta'',\theta)\left(\frac{p_0}{q_0}-\theta\right)^{k-l}+\frac{\partial P_l}{\partial y}\left(\frac{p_0}{q_0},\theta'\right) \left(\frac{p}{q}-\theta\right).
\end{equation}

Now
\[P_k(x,y)=\sum_{i=0}^{D}\sum_{j=0}^{e}{p_{ij}\binom{i}{k}x^{i-k}y^{j}}Ê\]
and therefore, using Lemma \ref{lem:Binom}, we find that
\begin{equation}\label{eq:xderibound}
\left|P_k\right| \leq \frac{2^D}{\sqrt{D}}\left|P\right| \leq \frac{2^D}{\sqrt{D}}c_1(2H)^{\alpha k}.
\end{equation}
Similarly we have
\[ \frac{\partial P_l}{\partial y}(x,y)=\sum_{i=0}^{D}\sum_{j=0}^{e}{p_{ij}\binom{i}{l}jx^{i-l}y^{j-1}}\]
and therefore
\begin{equation}\label{eq:yderibound}
\left|\frac{\partial P_l}{\partial y}\right| \leq e\frac{2^D}{\sqrt{D}}\left|P\right| \leq e\frac{2^D}{\sqrt{D}}c_1(2H)^{\alpha k}.
\end{equation}

Furthermore, we know that $\left|\theta\right| \leq H^d$. Therefore, $\left|\theta'\right|$, $\left|\theta''\right|$ and $\left|\frac{p_0}{q_0}\right|$ are all bounded from above by $H^d+1 \leq 2H^d$. It follows that
\[Ê\left|P_k(\theta'',\theta)\right| \leq (D+1)(e+1)\left|P_k\right|(2H^d)^{D}(H^d)^e. \]
This implies together with \eqref{eq:xderibound} and $\frac{D+1}{\sqrt{D}} \leq 2\sqrt{D}$ that
\[\left|P_k(\theta'',\theta)\right| \leq \sqrt{D}2^{2D}2(e+1)c_1(2H)^{\alpha k}H^{Dd+ed}.\]
Since $D \leq \delta k$ and $\beta = d\delta+\alpha$, we get
\[\left|P_k(\theta'',\theta)\right|Ê\leq \sqrt{\delta k}2^{2\delta k}2(e+1)c_1H^{ed}2^{\alpha k}H^{\beta k}.\]
Similarly, it follows from \eqref{eq:yderibound} that
\[ \left|\frac{\partial P_l}{\partial y}\left(\frac{p_0}{q_0},\theta'\right)\right| \leq \sqrt{\delta k}2^{2\delta k}2^ee2(e+1)c_1H^{ed}2^{\alpha k}H^{\beta k}.Ê\]

Since $2^ee \leq 2^{2e} \leq 2^{\frac{2\gamma}{d}k} \leq 2^{\frac{2}{3}k} \leq \frac{2^k}{\sqrt{k}}$ and $\binom{k}{l} \leqÊ\frac{2^k}{\sqrt{k}}$ by Lemma \ref{lem:Binom}, we conclude from \eqref{eq:meanvaluefull} that
\[Ê\left|Ê\lambda \right| \leq \sqrt{\delta}2^{(2\delta+1)k}2(e+1)c_1H^{ed}2^{\alpha k}H^{\beta k}\left\{\left|\theta-\frac{p_0}{q_0}\right|^{k-l}+\left|\theta-\frac{p}{q}\right|\right\}.\]
We have
\[Ê\sqrt{\delta}2^{(2\delta+1)k} \leq 2^{3\delta k} \leq 2^{d\delta k},\]
since $\delta \geq 1$ and therefore $\sqrt{\delta} \leq 2^{\delta-1} \leq 2^{(\delta-1)k}$. Also by \eqref{eq:newlbound} we have $k-l \geq \gamma k - ed$, and now the lemma follows on recalling \eqref{eq:lambdalowerbound}.
\end{proof}

We are now ready to prove the main theorem.

\begin{proof} (of the main theorem) By assumption, there are integers $p_0$ and $q_0 \geq 1$ such that
\begin{equation}\label{eq:lambdaineq}
\LambdaÊ= (2H)^{-\beta}q_0^{-\delta}\left|\theta-\frac{p_0}{q_0}\right|^{-\gamma} > 1.
\end{equation}
Let $p$ and $q \geq 1$ be any integers. Since $\beta > 0$, $\gamma > 0$ and $\delta > 0$, it follows from \eqref{eq:lambdaineq} that $\left|Ê\theta - \frac{p_0}{q_0}Ê\right| < 1$. This and \eqref{eq:lambdaineq} together with $ e \leq \kappa$ imply that $C \geq 1$, so we can restrict ourselves to pairs $(p,q)$ with $\left| \theta-\frac{p}{q} \right| < 1$. Therefore we can apply Lemma \ref{lem:upperestimate} and obtain
\[Ê\left| \theta -Ê\frac{p}{q} \right| \geq c_3^{-1}q_0^{-\delta k}q^{-e}(2H)^{-\beta k}-\left|\theta-\frac{p_0}{q_0}\right|^{\gamma k - ed}Ê\]
for any $k \geq \frac{ed}{\gamma}$. Using \eqref{eq:lambdaineq}, we can write this as
\begin{equation}\label{eq:lowerboundone}
Ê\left| \theta -Ê\frac{p}{q} \right| \geq X^{-1}\left(2-2\Lambda^{-k}c_3q^e\left|\theta-\frac{p_0}{q_0}\right|^{-ed}\right)
\end{equation}
with
\begin{equation}\label{eq:xdef}
X=2c_3q_0^{\delta k}q^e(2H)^{\beta k}=\tilde{c}c^kq^e.
\end{equation}

Since $\Lambda > 1$, we can fix $k$ now as the least integer $k \geq \frac{ed}{\gamma}$ with
\begin{equation}\label{eq:kbound}
\Lambda^k \geq 2c_3q^e\left|\theta-\frac{p_0}{q_0}\right|^{-ed}=\tilde{c}q^e\left|\theta-\frac{p_0}{q_0}\right|^{-ed}
\end{equation}
and then \eqref{eq:lowerboundone} implies that
\begin{equation}\label{eq:lowerboundfinal}
Ê\left| \theta -Ê\frac{p}{q} \right| \geq \frac{1}{X}.
\end{equation}

Now suppose first that $k-1 < \frac{ed}{\gamma}$. Then we deduce from \eqref{eq:kbound} that
\[Êq \leq \tilde{c}^{-\frac{1}{e}}\Lambda^{\frac{k}{e}}\left|\theta-\frac{p_0}{q_0}\right|^{d} \leq \tilde{c}^{-\frac{1}{e}}\Lambda^{\frac{d}{\gamma}+\frac{1}{e}}\left|\theta-\frac{p_0}{q_0}\right|^{d},\]
which implies together with $\tilde{c} \geq 1$ and \eqref{eq:lambdaineq} that
\[Êq \leq \Lambda^{\frac{d}{\gamma}+\frac{1}{e}}\left|\theta-\frac{p_0}{q_0}\right|^{d}=((2H)^{-\beta}q_0^{-\delta})^{\frac{d}{\gamma}+\frac{1}{e}}\left|\theta-\frac{p_0}{q_0}\right|^{-\frac{\gamma}{e}}.\]
If $\left|\theta-\frac{p_0}{q_0}\right| \geq 1$, this clearly implies that $q<1$, a contradiction. Thus, we can assume that $\left|\theta-\frac{p_0}{q_0}\right| < 1$ and hence $\left|p_0\right| < (\left|\theta\right|+1)q_0$.

Using the properties \eqref{eq:heightone}, \eqref{eq:heighttwo}, and \eqref{eq:heightthree} of the height, we find that
\[Ê\left|\theta-\frac{p_0}{q_0}\right| \geq H\left(\theta-\frac{p_0}{q_0}\right)^{-d} \geq (2H\max\{\left|p_0\right|,q_0\})^{-d}.\]
Together with the above, we deduce that
\[Ê\left|\theta-\frac{p_0}{q_0}\right| \geq (2H(\left|\theta\right|+1)q_0)^{-d} \geq (4H^{d+1}q_0)^{-d},\]
as $\left|\theta\right| \leq H^d$. It follows that
\[Êq \leq ((2H)^{-\beta}q_0^{-\delta})^{\frac{d}{\gamma}+\frac{1}{e}}(4H^{d+1}q_0)^{\frac{d\gamma}{e}} \leq ((2H)^{-\beta}q_0^{-\delta})^{d}(4H^{d+1}q_0)^{d},\]
since $\gamma < 1$ and $e \geq 1$. But this implies that
\[q \leq \left(2^{2-\beta}H^{d+1-\beta}q_0^{1-\delta}\right)^{d} \leq \left(2^{2-\beta}H^{2d-\beta}q_0^{1-\delta}\right)^{d}<1,\]
since $\deltaÊ\geq 1$ and $\beta \geq 2d\delta \geq 2d > 2$, and we get a contradiction.

We conclude that $k-1 \geq \frac{ed}{\gamma}$, and now it follows from the minimality of $k$ in \eqref{eq:kbound} that
\[Ê\Lambda^{k-1} < \tilde{c}q^e\left|\theta-\frac{p_0}{q_0}\right|^{-ed}.\]
We get
\[Ê(k-1)\log \Lambda < e \log q + \log c_4\]
with
\[Êc_4=\tilde{c}\left|\theta-\frac{p_0}{q_0}\right|^{-ed}.\]
It follows that
\[Êk < \frac{eÊ\log q}{\log \Lambda}+\frac{\log c_4}{\log \Lambda}+1.\]
Thus by \eqref{eq:xdef} we have
\[Ê\log X < \log \tilde{c} +e\log q+\left(\frac{eÊ\log q}{\log \Lambda}+\frac{\log c_4}{\log \Lambda}+1\right)\log c.\]
Rearranging terms, we get
\[Ê\log X < \log \tilde{c}+\log c+\frac{\log c}{\logÊ\Lambda}\log c_4+\kappaÊ\log q\]
and after exponentiating the theorem now follows from \eqref{eq:lowerboundfinal}, since
\[\frac{\log c}{\log \Lambda}=\frac{\kappa}{e}-1.\]
\end{proof}

\section{Examples}\label{sec:Examples}

We now study the family of polynomials
\[A(t)=(t-a)Q(t)+P(t),\]
parametrized by $a \in \mathbb{C}$ for fixed $P$ and $Q$, where $Q$ is monic. We will see that under certain conditions one zero of $A$ is exceptionally well approximated by $a$ so that we can apply our main theorem --- for this application we will later assume that $a$ is an integer and that $P$ and $Q$ are coprime and have integer coefficients. The following lemma supplies all the technical information we need to know about this family.

\begin{lem}\label{lem:rouche}
Let $P$ and $Q$ be polynomials in $\mathbb{C}[t]$, where $P$ is of degree $d_0$ and $Q$ is monic of degree $d-1 > d_0$, and let $a$ be any complex number. Put
\[R=\max_{Q(\theta)=0}{\left|\theta\right|}\]
and $A(t)=(t-a)Q(t)+P(t)$. Suppose that
\[\left|a\right| \geq \max\{1,L(P)\}\max\left\{2^{\frac{d_0}{d-d_0-1}},(R+1)^{d_0}\right\}+2R+2.\]
Then the following hold:
\begin{enumerate}[label=(\alph*)]
\item Counting multiplicities, the polynomial $A$ has exactly $d-1$ zeroes $\xi_i$ with $\left|\xi_i\right|<R+1$ ($i=1,\hdots,d-1$) and one zero $\xi$ with $\left|\xi-a\right|<1$, different from the $\xi_i$ ($i=1,\hdots,d-1$).
\item If $P$ and $Q$ are coprime, have integer coefficients and $a$ is an integer, then $A$ is irreducible over $\mathbb{Q}$.
\end{enumerate}
\end{lem}

\begin{proof}
We first prove (a). Put $f(z)=A(z)-P(z)=(z-a)Q(z)$ and $g(z)=P(z)$, both entire functions. Let $K_0$ be the open disk of radius $R+1$ around 0 and $K_1$ the open disk of radius 1 around $a$. We want to apply Rouch\'{e}'s theorem and deduce that $f$ and $f+g=A$ have the same number of zeroes (with multiplicities) in these disks. For this we need to show that $\left|f(z)\right| > \left|g(z)\right|$ for every $z$ on the contour of one of them.

For $z \in \partial K_0$ we have $\left|Q(z)\right| \geq 1$, as $Q$ is monic and
\[Ê\left|z-\theta\right| \geq \left|z\right|-\left|\theta\right| \geq \left|z\right|-R=1\]
for every zero $\theta$ of $Q$. We deduce that
\[ \left|f(z)\right| \geq \left|z-a\right| \geq \left|a\right|-R-1.\]
As $\left|a\right|-R-1 > \max\{1,L(P)\}(R+1)^{d_0}$, it follows that
\[Ê\left|f(z)\right| > \max\{1,L(P)\}(R+1)^{d_0} \geq L(P)(R+1)^{d_0} \geq \left|g(z)\right|\]
as required.

For $z \inÊ\partial K_1$ we have
\[\left|f(z)\right|=\left|Q(z)\right| \geq (\left|a\right|-R-1)^{d-1},\]
since
\[\left|z-\theta\right| \geq \left|z\right|-\left|\theta\right| \geq (\left|a\right|-1)-R\]
for every zero $\theta$ of $Q$ and $Q$ is monic. Hence we get
\[Ê\left|f(z)\right| \geq (\left|a\right|-R-1)^{d-1}.\]

Using $\left|a\right| \geq 2R+3$, which implies that $\left|a\right|-R-1 \geq \frac{\left|a\right|+1}{2}$, we find that
\[ \left|f(z)\right| \geq 2^{-d_0}(\left|a\right|-R-1)^{d-d_0-1}(\left|a\right|+1)^{d_0}.\]
Since $\left|a\right|-R-1 > \max\{1,L(P)\}2^{\frac{d_0}{d-d_0-1}}$ and $d-d_0-1 \geq 1$, it follows that
\[\left|f(z)\right| > \max\{1,L(P)\}^{d-d_0-1}(\left|a\right|+1)^{d_0} \geq L(P)(\left|a\right|+1)^{d_0} \geq \left|g(z)\right|\]
as required.

We can therefore apply Rouch\'{e}'s theorem and deduce that $f$ and $f+g=A$ have the same number of zeroes (with multiplicities) in $K_0$ and $K_1$. The disk $K_0$ is disjoint from the disk $K_1$, since $\left|a\right|-1 > R+1$. Out of the zeroes of $f$, $K_0$ therefore contains exactly the $d-1$ zeroes of $Q$ (counted with multiplicities), and $K_1$ contains exactly the zero $a$. Hence (a) follows.

In order to show (b), assume that $A$ is reducible over $\mathbb{Q}$, so $A=A_1A_2$ with $A_1$, $A_2$ in $\mathbb{Q}[t]\backslash\mathbb{Q}$. Since $\xi$ is a simple zero of $A$, it follows that there is $B \in \{A_1,A_2\}$ with $B\left(\xi\right) \neq 0$. Let $\omega$ be a zero of $B$ and hence a zero of $A$, not equal to $\xi$. It follows that $\left|\omega\right| < R+1$, so in particular $\omegaÊ\neq a$. The polynomial $A$ is monic, since $Q$ is monic and $d > d_0$, and hence $\omega$ is an algebraic integer. It follows that
\[Ê\alpha=\frac{P(\omega)}{\omega-a}=-Q(\omega) \in \mathbb{Q}(\omega)\]
is an algebraic integer as well.

Let $\sigma: \mathbb{Q}(\omega) \to \mathbb{C}$ be an embedding. Then $\sigma(\omega)$ is a zero of $B$, so $\left|\sigma(\omega)\right| <R+1$ and it follows that
\[ \left|\sigma(\alpha)\right|=\left|\frac{P(\sigma(\omega))}{\sigma(\omega)-a}\right| \leq \frac{L(P)(R+1)^{d_0}}{\left|a\right|-R-1}<1.\]
Since this holds for every embedding $\sigma$ of $\mathbb{Q}(\omega)$ and $\alpha \in \mathbb{Q}(\omega)$ is an algebraic integer, it follows that $\alpha=0$ and hence $P(\omega)=Q(\omega)=0$, which contradicts the fact that $P$ and $Q$ are coprime. Hence $A$ is irreducible.
\end{proof}

We now take $a \in \mathbb{Z}$ and $P$, $Q$ coprime in $\mathbb{Z}[t]$ with $Q$ of degree $d-1 > d_0$. If $\left|a\right|$ is large enough, we can apply both parts of the lemma to $A$; it follows that $\xi$ is real (otherwise the complex conjugate $\bar{\xi} \neq \xi$ would also be a zero of $A$ with $\left|\bar{\xi}-a\right|<1$), algebraic of degree $d$, and that
 \begin{equation}\label{eq:heightxi}
 \left|\xi\right|^{\frac{1}{d}}Ê\leq H(\xi) \leq (R+1)\left|\xi\right|^{\frac{1}{d}},
 \end{equation}
 since $\left|\xi\right| > \left|a\right|-1 \geq 1$.
 
\begin{thm}\label{thm:effectivelower}
Let $P$ and $Q$ be coprime polynomials with integer coefficients, where $P$ is of degree $d_0$ and $Q$ is monic of degree $d-1>0$, and let $a$ be any integer. Put
\[R=\max_{Q(\theta)=0}{\left|\theta\right|}\]
and $A(t)=(t-a)Q(t)+P(t)$. Suppose further that
\[ d \geq \hat{d}=2.13d_0+23,\]
\[\left|a\right| \geq L(P)\max\left\{2,(R+1)^{d_0}\right\}+2R+2\]
and $\eta > 0$.

Then there is a unique real zero $\xi$ of $A$ with $\left|\xi-a\right| < 1$, and it is of effective strict type at most $\kappa = \hat{\kappa}+\eta$ for
\[\left|a\right|Ê\geq a_0=a_0(d,\eta,R,P)=\left(2^{2d}(R+1)^{d}L(P)^{2}\right)^{\frac{\kappa}{\eta(11f(\hat{d})-1)}}+1,\]
where
\[Ê\hat{\kappa} = 10\left(1+\frac{1}{11f(d)-1}\right) \leq 10\left(1+\frac{1}{11f(\hat{d})-1}\right)<\hat{d}\]
and
\[f(u)=\left(3-2\sqrt{2}\right)\frac{u-d_0-1}{u+\sqrt{2}-1}.\]

More precisely
\[Ê\left|\xi-\frac{p}{q}\right| \geq \frac{1}{Cq^{\kappa}}\]
for all integers $p$ and $q \geq 1$, where
\[ÊC = C(d,d_0,\eta,R,a)= 2^{42(d_0+\eta+1)d^2}(R+1)^{5(d_0+\eta+1)d^2}\left|a\right|^{14(d_0+\eta+1)d^2}.\]
\end{thm}

Since $P$ and $Q$ are coprime, we have $P \neq 0$ and hence $L(P) \geq 1$. Since further $\frac{d_0}{d-d_0-1} \leq 1$ for $dÊ\geq \hat{d}$ and $d-1Ê\geq \hat{d}-1 > d_0$, our polynomial $A$ satisfies all conditions in Lemma \ref{lem:rouche}. We see that $f(u)$ is monotonically increasing for $u \geq d_0+2$. As we furthermore have
\[\lim_{u \to \infty}{f(u)}=3-2\sqrt{2},\]
the upper bound for the effective strict type of $\xi$ tends to
\[Ê\frac{55}{14}(4+\sqrt{2})Ê\approx 21.2701247\hdots\]
with $dÊ\to \infty$ and $\etaÊ\to 0$ (which means that $\left|a\right| \to \infty$).

Our choice of $\hat{d}$ guarantees that the main theorem can be applied, if $\left|a\right|$ is large enough, and that $\hat{\kappa} < \hat{d} \leq d$ so that our theorem is an improvement over Liouville's for $\eta$ small enough. We haven't chosen the smallest possible such $\hat{d}$, so there might be better choices for any given $d_0$. Certainly, the factor $2.13$ is, although reasonably good, not best possible for $d_0 \to \infty$.

\begin{proof}
We want to use the main theorem with $p_0=a$ and $q_0=1$. We will choose the parameters $e$ and $\epsilon$ later in a nearly optimal way. It follows from \eqref{eq:heightxi} that
\begin{equation}\label{eq:cdef}
2^{\beta}\left|\xi\right|^{\frac{\beta}{d}} \leq c \leq 2^{\beta}(R+1)^{\beta}\left|\xi\right|^{\frac{\beta}{d}}.
\end{equation}
We know that $Q(\xi) \neq 0$ as otherwise $Q(\xi)=P(\xi)=0$ and $\xi$ would be a common zero of $P$ and $Q$. Since $\left|\xi\right| > \left|a\right|-1 \geq 1$ and $\left|\xi\right| > \left|a\right|-1 > 2R$, it follows that
\begin{equation}\label{eq:diffdef}
\left|\xi-a\right|=\left|\frac{P(\xi)}{Q(\xi)}\right| \leq \frac{L(P)\left|\xi\right|^{d_0}}{(\left|\xi\right|-R)^{d-1}} < L(P)2^{d-1}\left|\xi\right|^{-(d-d_0-1)}.
\end{equation}

Furthermore, if $b$ is the (non-zero) leading coefficient of $P \neq 0$, then it is well-known that every zero $\theta$ of $P$, which is therefore also a zero of the monic polynomial $\frac{1}{b}P$, satisfies $\left|\theta\right| \leq L\left(\frac{1}{b}P\right)=\frac{L(P)}{\left|b\right|}$. If we put $R'=\max_{P(\theta)=0}{\left|\theta\right|}$, if $d_0 > 0$, and $R'=0$, if $d_0=0$, it follows that $R' \leq \frac{L(P)}{\left|b\right|}$. Since $P$ has integer coefficients, we have $\left|b\right|Ê\geq 1$ and therefore $R' \leq L(P)$. As $\left|\xi\right| > \left|a\right|-1 > 2L(P)$ and $\left|\xi\right| > R$, it follows that
\begin{equation}\label{eq:difflower}
\left|\xi-a\right|=\left|\frac{P(\xi)}{Q(\xi)}\right| \geq \frac{\left|b\right|(\left|\xi\right|-R')^{d_0}}{(\left|\xi\right|+R)^{d-1}} > 2^{-(d+d_0-1)}\left|\xi\right|^{-(d-d_0-1)}.
\end{equation}

We can now use \eqref{eq:cdef} and \eqref{eq:diffdef} to deduce that
\begin{equation}\label{eq:lambdadef}
\Lambda=c^{-1}\left|\xi-a\right|^{-\gamma} > 2^{-\beta-\gamma(d-1)}(R+1)^{-\beta}L(P)^{-\gamma}\left|\xi\right|^{\frac{\beta}{d}(\lambda-1)}=\Lambda_0,
\end{equation}
where 
\begin{equation}\label{eq:smalllambdadefone}
\lambda=\frac{\gamma d(d-d_0-1)}{\beta}=(d-d_0-1)(e+1)g_d(\epsilon)
\end{equation}
with
\[Êg_d(\epsilon)=\frac{\epsilon(1-\epsilon)}{(d+\epsilon)(1+\epsilon)}.\]

If $\Lambda_0 > 1$ and hence $\Lambda > 1$, it follows from the main theorem that $\xi$ is of effective strict type at most
\[Êe\left(1+\frac{\log c}{\log \Lambda}\right) \leq e\left(1+\frac{\log c}{\log \Lambda_0}\right).\]
By \eqref{eq:cdef} and \eqref{eq:lambdadef}, this is at most equal to
\[\kappa_a=e\left(1+\frac{dc_1+\log\left|\xi\right|}{-dc_1-\frac{\lambda(d-1)}{d-d_0-1}\log 2-\frac{\lambda}{d-d_0-1}\log L(P)+(\lambda-1)\log\left|\xi\right|}\right)\]
with $c_1=\log 2+\log(R+1)$. We see that $\kappa_a$ tends to $\kappa_{\infty}=e\left(1+\frac{1}{\lambda-1}\right)$ for $\left|\xi\right| \to \infty$ (or equivalently $\left|a\right| \to \infty$).

For given $e$ we want to choose $\epsilon$ such that $\kappa_{\infty}$ is minimized, which by \eqref{eq:smalllambdadefone} means that $g_d(\epsilon)$ is maximized. This function has a maximum in the interval $[0,1]$ at
\[\epsilon_0(d)=\frac{\sqrt{2d(d+1)}-d}{d+2}.\]
To simplify the following computations we choose 
\begin{equation}\label{eq:epsilondef}
\epsilon=\lim_{dÊ\to \infty}{\epsilon_0(d)}=\sqrt{2}-1
\end{equation}
and get
\begin{equation}\label{eq:smalllambdadef}
\lambda=(e+1)f(d).
\end{equation}

We need $\lambda > 1$ in order to be able to make sure that $\Lambda_0 > 1$, so we have to choose
\[e > \frac{1}{f(d)}-1.\]
We want to minimize
\[\kappa_\infty=e\left(1+\frac{1}{(e+1)f(d)-1}\right).\]
We choose $e=10$, which minimizes the expression
\[e\left(1+\frac{1}{(e+1)(3-2\sqrt{2})-1}\right)Ê\]
for $eÊ\in \mathbb{N}$ and therefore may be expected to give asymptotically the best $\kappa$. We get $\kappa_\infty=\hat{\kappa}$.

Since $f(u)$ is monotonically increasing for $u \geq d_0+2$, we see that
\[\frac{1}{f(d)}-1 \leq \frac{1}{f(\hat{d})}-1=\left(3+2\sqrt{2}\right)\left(1-\frac{d_0+\sqrt{2}}{\hat{d}+\sqrt{2}-1}\right)^{-1}-1.\]
As $\hat{d}+\sqrt{2}-1 \geq 2.13(d_0+\sqrt{2})$, it follows that
\[\left(3+2\sqrt{2}\right)\left(1-\frac{d_0+\sqrt{2}}{\hat{d}+\sqrt{2}-1}\right)^{-1}-1 \leq \left(3+2\sqrt{2}\right)\left(1-\frac{1}{2.13}\right)^{-1}-1 < 10\]
and hence
\begin{equation}\label{eq:guardsguards}
\frac{1}{f(d)}-1 \leq \frac{1}{f(\hat{d})}-1 < 10= e.
\end{equation}
We deduce that $\lambda=11f(d)>1$ as required.

Using \eqref{eq:lambdadef} and \eqref{eq:smalllambdadefone}, we now see that $\Lambda_0 > 1$ is equivalent to
\begin{equation}\label{eq:jingo}
\left|\xi\right|^{\lambda-1} > 2^{d+\frac{\lambda(d-1)}{(d-d_0-1)}}(R+1)^{d}L(P)^{\frac{\lambda}{d-d_0-1}}.
\end{equation}
As $\xi$ should be of effective strict type at most $\kappa=\hat{\kappa}+\eta$, we furthermore need to make sure that
\[Ê\kappa_a \leq \kappa_{\infty}+\eta=\hat{\kappa}+\eta=\kappa.\]
This is equivalent to
\[\frac{dc_1+\log\left|\xi\right|}{-dc_1-\frac{\lambda(d-1)}{d-d_0-1}\log 2-\frac{\lambda}{d-d_0-1}\log L(P)+(\lambda-1)\log\left|\xi\right|} \leq \frac{1}{\lambda-1}+\frac{\eta}{10}.\]
If $\xi$ satisfies \eqref{eq:jingo}, then the denominator of the left-hand side is positive, so we can multiply by it without changing the direction of the inequality. This in turn is equivalent to
\[ \frac{\kappa}{10} d c_1+\left(\frac{\kappa}{10}-1\right)c_2 \leq \frac{\eta(\lambda-1)}{10}\logÊ\left|\xi\right|\]
with
\[c_2=\frac{\lambda(d-1)}{d-d_0-1}\log2+\frac{\lambda}{d-d_0-1}\log L(P) ,Ê\]
since $\kappa=\hat{\kappa}+\eta=10(1+\frac{1}{\lambda-1})+\eta$. This inequality follows from
\begin{equation}\label{eq:thud}
\frac{\kappa}{10}(dc_1+c_2) \leq \frac{\eta(\lambda-1)}{10}\logÊ\left|\xi\right|.
\end{equation}

We see that
\[Êc_2=\lambda \log 2+\frac{\lambda d_0}{d-d_0-1}\log 2+\frac{\lambda}{d-d_0-1}\log L(P) \leq \lambda(2\log 2+\log L(P)),\]
since $d-d_0-1 \geq 1$ and $\frac{d_0}{d-d_0-1} \leq 1$. Since $f$ is monotonically growing and $\lim_{u \to \infty}{f(u)}=3-2\sqrt{2}$, it follows from \eqref{eq:smalllambdadef} that
\begin{equation}\label{eq:menatarms}
Ê\lambda \leq 33-22\sqrt{2} < 2.
\end{equation}
Hence we have
\[Êc_2 \leq (33-22\sqrt{2})(2\log 2+\log L(P)) \leq 2\log L(P)+4\log2.\]

Since $\frac{\eta(\lambda-1)}{10}>0$, we deduce that \eqref{eq:thud} follows in turn from
\[Ê\frac{\kappa}{\eta(\lambda-1)}(dc_1+2\log L(P)+4\log2) \leq \log \left|\xi\right|,\]
which is equivalent to
\[ 2^{\frac{\kappa(d+4)}{\eta(\lambda-1)}}(R+1)^{\frac{\kappa d}{\eta(\lambda-1)}}L(P)^{\frac{2\kappa}{\eta(\lambda-1)}}Ê\leq \left|\xi\right|.\]
As $dÊ\geq 23$, $\left|\xi\right|>\left|a\right|-1$ and $\lambda-1 \geq 11f(\hat{d})-1$, this follows from $\left|a\right| \geq a_0$.

Using $\frac{\kappa}{\eta}=1+\frac{\hat{\kappa}}{\eta}>1$, we likewise deduce from $\left|a\right| \geq a_0$ that
\[Ê\left|\xi\right|^{\lambda-1}>2^{2d}(R+1)^dL(P)^2.\]
But since $\frac{d-1}{d-d_0-1} \leq 2$, $d-d_0-1 \geq 1$, $d \geq 23$ and $\lambda < 2$ by \eqref{eq:menatarms}, this implies \eqref{eq:jingo} and hence $\Lambda_0 > 1$.

It now follows from the main theorem that $\xi$ is of effective strict type at most
\[Ê10\left(1+\frac{\log c}{\log \Lambda}\right) \leq 10\left(1+\frac{\log c}{\logÊ\Lambda_0}\right) \leq \kappa_a \leq \kappa\]
and it remains to estimate $\hat{\kappa}$ and show that $C$ can be bounded by an expression of the required form.

We estimate $\hat{\kappa}$ first. Since $f$ is monotonically increasing, we have
\[Ê\hat{\kappa} \leq 10\left(1+\frac{1}{11f(\hat{d})-1}\right)\]
as required in the theorem. It remains to show that
\begin{equation}\label{eq:feetofclay}
\hat{d}-10\left(1+\frac{1}{11f(\hat{d})-1}\right) > 0.
\end{equation}
If we multiply the left-hand side first by its denominator $11f(\hat{d})-1$, which is positive because of \eqref{eq:guardsguards}, and then by the denominator $\hat{d}+\sqrt{2}-1$ of $f(\hat{d})$, which is obviously positive as well, we get a polynomial in $\hat{d}$ and $d_0$. After the substitution $\hat{d}=2.13d_0+23$, this ends up as $Kd_0^2+Ld_0+M$ with positive $K$, $L$, $M$. This verifies \eqref{eq:feetofclay} and so retroactively $\hat{\kappa} < \hat{d}$.

We proceed to bound $C$. It follows from \eqref{eq:epsilondef} that
\[\frac{\alpha}{2\delta}=\frac{d}{2\epsilon}=\frac{1+\sqrt{2}}{2}d\]
and
\[\frac{e\beta}{\delta}=10d\left(1+\frac{1}{\epsilon}\right)=10(2+\sqrt{2})d.\]
Using this together with \eqref{eq:heightxi}, we estimate $\tilde{c}$ as
\[\tilde{c} \leq 2^{\frac{1+\sqrt{2}}{2}d+12}11^{\frac{1+\sqrt{2}}{2}d+1}(R+1)^{10(2+\sqrt{2})d}\left|\xi\right|^{10(2+\sqrt{2})}.\]
Together with $d \geq 23$, this implies that
\begin{equation}\label{eq:ctildedef}
\tilde{c} \leq 2^{0.27d^2}(R+1)^{1.49d^2}\left|\xi\right|^{0.07d^2}.
\end{equation}

Since $\lambda> 1$ and $\gamma < 1$, we can estimate
\[\beta=\frac{\gamma d(d-d_0-1)}{\lambda} < d(d-d_0-1),\]
using \eqref{eq:smalllambdadefone}. Hence, it follows from \eqref{eq:cdef} that
\begin{equation}\label{eq:cbetterestimate}
c \leq 2^{d(d-d_0-1)}(R+1)^{d(d-d_0-1)}\left|\xi\right|^{d-d_0-1} \leq 2^{d^2}(R+1)^{d^2}\left|\xi\right|^{d}.
\end{equation}

Since $\tilde{c} \geq 1$, $e\left(1+\frac{\log c}{\log \Lambda}\right) \leq \kappa$ and $\left|\xi-a\right|<1$, we have
\begin{equation}\label{eq:bigcthefirst}
C \leq c\tilde{c}^{\frac{\kappa}{e}}\left|\xi-a\right|^{d(e-\kappa)},
\end{equation}
where $\kappa$ is defined as in this theorem and not as in the main theorem.

Using \eqref{eq:guardsguards} and $\eta > 0$, we see that 
\[Ê\kappa-10 = 10\left(1+\frac{1}{11f(d)-1}\right)+\eta-10=\frac{10}{11f(d)-1}+\eta>0.\]
This implies together with \eqref{eq:difflower} and \eqref{eq:bigcthefirst} that
\[C \leq c\tilde{c}^{\frac{\kappa}{10}}\left|\xi-a\right|^{d(10-\kappa)} \leq c\tilde{c}^{\frac{\kappa}{10}}2^{d(d+d_0-1)(\kappa-10)}\left|\xi\right|^{d(d-d_0-1)(\kappa-10)}.\]
Inserting $\kappa=\hat{\kappa}+\eta < \hat{d}+\eta=2.13d_0+23+\eta$ into this inequality and using $d-d_0-1 \leq d$ and $d+d_0-1 \leq 2d$ as well as $\tilde{c} \geq 1$, we get
\[ÊC \leq c\tilde{c}^{\frac{2.13d_0+23+\eta}{10}}2^{2d^2(2.13d_0+13+\eta)}\left|\xi\right|^{d^2(2.13d_0+13+\eta)}\]
and hence
\[C \leq c\tilde{c}^{2.3(d_0+\eta+1)}2^{26(d_0+\eta+1)d^2}\left|\xi\right|^{13(d_0+\eta+1)d^2}.\]
Inserting \eqref{eq:ctildedef} and \eqref{eq:cbetterestimate} into this inequality and using $d \geq 23$ yields
\[CÊ\leq 2^{28(d_0+\eta+1)d^2}(R+1)^{5(d_0+\eta+1)d^2}\left|\xi\right|^{14(d_0+\eta+1)d^2}.\]
Using $\left|\xi\right| < \left|a\right|+1 < 2\left|a\right|$, we deduce the theorem (with a new $C$).
\end{proof}

We present two immediate corollaries, obtained by specializing $P$ and $Q$.

\begin{cor}\label{cor:bombieripoly}
Suppose that $dÊ\geq 23$, $0 < \eta \leq 1$ and
\[A(t)=t^{d}-at^{d-1}\pm1\]
for an integer $a$ with $\left|a\right| \geq 4$. Then there is a unique real zero $\xi$ of $A$ which satisfies ${\left|\xi-a\right|<1}$, and it is of effective strict type at most $\kappa = \hat{\kappa}+\eta$ for
\[\left|a\right| \geq a_0(d,\eta)=2^{\frac{2.6 \kappa d}{\eta}}+1,\]
where
\[Ê\hat{\kappa} = 10\left(1+\frac{1}{11f(d)-1}\right)< 22.94\]
and
\[f(u)=\left(3-2\sqrt{2}\right)\frac{u-1}{u+\sqrt{2}-1}.\]
More precisely
\[Ê\left|\xi-\frac{p}{q}\right| \geq \frac{1}{Cq^{\kappa}}\]
for all integers $p$ and $q \geq 1$ with
\[ÊC=C(d,a)=2^{84d^2}\left|a\right|^{28d^2}.\]
\end{cor}

The result that we quoted in Section \ref{sec:Introduction} now directly follows from this theorem by choosing $\eta=0.05$ and noting that
\begin{equation}\label{eq:aboundonazero}
a_0(d,0.05) \leq 2^{\frac{2.6Ê\cdot 22.99d}{0.05}}+1 \leq 2^{1196d}.
\end{equation}
The bound for the effective strict type asymptotically tends to the same limit as in Theorem \ref{thm:effectivelower}. Cf. Bombieri's results for $A(t)=t^d-at^{d-1}+1$, namely Section V, Example 3 in \cite{B82}, where he showed that the effective strict type of any $\zeta$ which generates $\mathbb{Q}(\xi)$ over $\mathbb{Q}$ is at most $39.2574$, if $d \geq 40$ and $a \geq a_0(d)$, as well as one of the applications in \cite{B87}, where he showed that the effective strict type of any irrational $\zeta$ in $\mathbb{Q}(\xi)$ is at most $13.209446$, if $d$ is large enough and $a \geq a_0(d)$. However, explicit values for both $C$ and $a_0$ are only given for $d=200$ in Example 2 of \cite{B82} with $\kappa=50$.

\begin{proof}
We apply Theorem \ref{thm:effectivelower} with $Q(t)=t^{d-1}$ and $P(t)=\pm1$, so $R=d_0=0$ (thus our $f$ coincides with the previous $f$) and $L(P)=1$. The corollary follows, since $\eta \leq 1$, $11f(\hat{d})-1=11f(23)-1 \geq \frac{1}{1.3}$ and
\[Ê\hat{\kappa} \leq 10\left(1+\frac{1}{11f(23)-1}\right) < 22.94.\]
\end{proof}

\begin{cor}\label{cor:combieripoly}
Suppose that $dÊ\geq 23$, $d$ odd, $0 < \eta \leq 1$ and 
\[A(t)=(t-a)(t^{2}+1)^{\frac{d-1}{2}}\pm1\]
for an integer $a$ with $\left|a\right| \geq 6$. Then there is a unique real zero $\xi$ of $A$ which satisfies ${\left|\xi-a\right|<1}$, and it is of effective strict type at most $\kappa = \hat{\kappa}+\eta$ for
\[\left|a\right| \geq a_0(d,\eta)=2^{\frac{3.9 \kappa d}{\eta}}+1,\]
where
\[Ê\hat{\kappa} =10\left(1+\frac{1}{11f(d)-1}\right) < 22.94\]
and
\[f(u)=\left(3-2\sqrt{2}\right)\frac{u-1}{u+\sqrt{2}-1}.\]
More precisely
\[Ê\left|\xi-\frac{p}{q}\right| \geq \frac{1}{Cq^{\kappa}}\]
for all integers $p$ and $q \geq 1$ with
\[ÊC=C(d,a)=2^{94d^2}\left|a\right|^{28d^2}.\]
\end{cor}

\begin{proof}
We apply Theorem \ref{thm:effectivelower} with $Q(t)=(t^2+1)^{\frac{d-1}{2}}$ and $P(t)=\pm1$, so $R=1$, $d_0=0$ (thus our $f$ coincides with the previous $f$) and $L(P)=1$. The corollary follows, since $\eta \leq 1$, $11f(\hat{d})-1=11f(23)-1 \geq \frac{1}{1.3}$ and
\[Ê\hat{\kappa} \leq 10\left(1+\frac{1}{11f(23)-1}\right) < 22.94.\]
\end{proof}

Analogous results with the same upper bound for the effective strict type hold for polynomials of the type
\[ A(t)=(t-a)\prod_{i=1}^{m}{(t-a_i)}\prod_{j=1}^{n}{(t^2+b_jt+c_j)}\pm1,\]
where $m+2n \geq 22$, $a_1$, \dots, $a_m$ are distinct integers and $(b_1,c_1)$, \dots, $(b_n,c_n)$ are distinct pairs of integers with $b_j^2-4c_j < 0$ ($j=1,\hdots,n$).

Such polynomials define the so-called ABC fields, which are named after Ankeny, Brauer and Chowla, who used them in \cite{ABC56} to construct number fields with large class numbers compared to their discriminant. Our irreducibility proof in Lemma \ref{lem:rouche} is inspired by a similar argument in that article.

\section{Applications}\label{sec:Applications}

We will now apply Corollary \ref{cor:bombieripoly} to derive explicit bounds for the solutions of the corresponding Diophantine equation
\[ x^d-ax^{d-1}y+y^d=m.\]
We will deduce that its solutions grow at most polynomially in $\left|a\right|$.

Of course, we need $\left|a\right| \geq a_0$ in order to apply Corollary \ref{cor:bombieripoly}. For $\left|a\right| < a_0$ we use the following theorem, which was proven in a more refined form by Bugeaud and Gy\"{o}ry with Baker's method of linear forms in logarithms and is far more general than our theorem will be, but whose upper bound depends exponentially rather than polynomially on the height of the form.

\begin{thm}\label{thm:Baker}
Suppose that $F(x,y)$ is an irreducible binary form of degree $d \geq 3$ and with integer coefficients having absolute values at most $\mathcal{H}$. Let $m$ denote any integer. Then all solutions of 
\[ÊF(x,y)=m\]
in integers $x$ and $y$ satisfy
\[Ê\max\{\left|x\right|,\left|y\right|\} \leq \exp\left(d^{40d}\mathcal{H}^{4d}\right)\left|m\right|^{d^{40d}\mathcal{H}^{4d}}.\]
\end{thm}

\begin{proof}
If $m=0$, the theorem is obvious. If $m \neq 0$, we apply \cite{BG96}, Theorem 3, pp.\ 275sq., with $n=d$, $H=3\mathcal{H}$ ($ \geq 3$, as $\mathcal{H} \geq 1$), $b=m$ and $B=e\left|m\right|$ ($ \geq e=\exp(1)$). It follows that any solution of $F(x,y)=m$ in integers $x$ and $y$ satisfies
\[Ê\max\{\left|x\right|,\left|y\right|\} < \exp\left(c_4 H^{2d-2}(\log H)^{2d-1} \log B\right)\]
with $c_4=c_4(d)=3^{3d+27}d^{18d+18}$. Since $d \geq 3$, we have $c_4 \leq d^{21d+45} \leq d^{36d}$
as well as 
\[ÊH^{2d-2}(\log H)^{2d-1} \leq H^{4d} = 3^{4d}\mathcal{H}^{4d} \leq d^{4d}\mathcal{H}^{4d}\]
and the theorem follows.
\end{proof}

We are almost ready to state our theorem, but there is one other obstacle: As the corollary gives a lower bound for the approximability of only one special zero $\xi$ of $A(t)$ by rationals, we need an additional assumption to make sure that all the other zeroes are non-real so that we can use \eqref{eq:imaginarypartbound} to deduce a lower bound for their approximability by rationals. In this case, this assumption is that $d$ is odd and $a < 0$.

\begin{thm}\label{cor:bombieriequation}
Let $d$, $a$ and $m$ be integers with $dÊ\geq 23$, $d$ odd and $a \leq -4$. Let $x$ and $y$ be any integers satisfying
\[Êx^d-ax^{d-1}y+y^d= m.\]
Then
\[Ê\max\{\left|x\right|,\left|y\right|\} \leq \left\{
	\begin{array}{ll}
		\left|a\right|^{\frac{29d^2}{d-22.99}}\left|m\right|^{\frac{1}{d-22.99}} & \mbox{if } \left|a\right| \geq 2^{1196d} \\
		\exp(2^{4824d^2})\left|m\right|^{2^{4824d^2}} & \mbox{if } \left|a\right| < 2^{1196d}
	\end{array}
\right. .\]
\end{thm}

\begin{proof}
Assume first that $y=0$. It follows that $x^{d} = m$ and hence $x$ and $y$ satisfy the bound. Therefore we can assume that $yÊ\neq 0$. Since $\left|a\right| \geq 3$, we can apply Lemma \ref{lem:rouche} with $Q(t)=t^{d-1}$ (so $R=0$) and $P(t)=1$ to $A(t)=t^d-at^{d-1}+1$.
Hence $A(t)$ is irreducible over $\mathbb{Q}$ and we can write
\[ m=A(x,y)=y^dA\left(\frac{x}{y}\right)=y^d\prod_{i=1}^{d}{\left(\frac{x}{y}-\xi_{i}\right)},\]
where $\left|\xi_i\right|<R+1$ for $i<d$ as in Lemma \ref{lem:rouche} and $\xi_d=\xi$ with $\left|\xi-a\right|<1$. It follows that
\begin{equation}\label{eq:rootsproduct}
\prod_{i=1}^{d}{\left|\frac{x}{y}-\xi_{i}\right|} =\frac{\left|m\right|}{\left|y\right|^d}
\end{equation}
and hence
\begin{equation}\label{eq:nearroot}
\left|\frac{x}{y}-\xi_{i}\right| \leq \frac{\left|m\right|^{\frac{1}{d}}}{\left|y\right|}
\end{equation}
for some $i$.

We first treat the case $i < d$: it follows that $\left|\xi_i\right|<1$. Now we have
\[ÊA(t)=(t-a)t^{d-1}+1Ê\geq 1\]
for all $t \in [-1,1]$, since $t-a \geq -a-1 \geq 3$ and $d-1$ is even, so $t^{d-1}$ is non-negative. Therefore, $\xi_i$ must be non-real.

It then follows from \eqref{eq:imaginarypartbound} with $\theta=\xi_i$ and \eqref{eq:nearroot} that
\[Ê\frac{1}{2}(2H(\xi_i)^2)^{-d^2}Ê\leq \left| \Im \xi_i \right| \leq \left|\frac{x}{y}-\xi_{i}\right| \leq \frac{\left|m\right|^{\frac{1}{d}}}{\left|y\right|}.\]
Using $H(\xi)=H(\xi_i)$, since $\xi=\xi_d$ and $\xi_i$ are conjugates, together with the bound \eqref{eq:heightxi} for $H(\xi)$ (where $R=0$) and $\left|\xi\right| < \left|a\right|+1$, we deduce that
 \[Ê\frac{1}{2}(2(\left|a\right|+1)^{\frac{2}{d}})^{-d^2}Ê\leq \frac{\left|m\right|^{\frac{1}{d}}}{\left|y\right|}\]
and hence
\[Ê\left|y\right| \leq 2^{d^2+1}(\left|a\right|+1)^{2d}\left|m\right|^{\frac{1}{d}}.\]

We use \eqref{eq:nearroot} again to get
\[Ê\left|x\right| \leq \left|m\right|^{\frac{1}{d}}+\left|\xi_iy\right| \leq (2^{d^2+1}(\left|a\right|+1)^{2d}+1)\left|m\right|^{\frac{1}{d}},\]
and therefore
\[\max\{\left|x\right|,\left|y\right|\} \leq 2^{d^2+2}(\left|a\right|+1)^{2d}\left|m\right|^{\frac{1}{d}}.\]
This bound is majorized by the bound in the theorem.

We now treat the case $i=d$, so that $\xi_i=\xi$. We have either $\left| yÊ\right| \leq 2\left|m\right|^{\frac{1}{d}}$ and then \eqref{eq:nearroot} implies as above that $x$ and $y$ satisfy the bound, since $\left|\xi_i\right|=\left|\xi\right|<\left|a\right|+1$, or we have $\left|y\right| > 2\left|m\right|^{\frac{1}{d}}$ and then \eqref{eq:nearroot} implies that
\[Ê\left|\frac{x}{y}-\xi\right| < \frac{1}{2}.\]
Since
\[Ê\left|\frac{x}{y}-\xi_i\right| \geq \left|\xi-\xi_i\right|-\left|\frac{x}{y}-\xi\right| > \left|\xi\right|-\left|\xi_i\right|-\frac{1}{2} > (\left|a\right|-1)-\frac{3}{2}\]
for $i=1,\hdots,d-1$ and since $\left|a\right| \geq 4$ implies that $\left|a\right|-\frac{5}{2} > 1$, we can use \eqref{eq:rootsproduct} to get
\[\left|\frac{x}{y}-\xi\right| = \frac{\left|m\right|}{\left|y\right|^d\prod_{i=1}^{d-1}{\left|\frac{x}{y}-\xi_i\right|}} \leq \frac{\left|m\right|}{\left(\left|a\right|-\frac{5}{2}\right)^{d-1}\left|y\right|^d} < \frac{\left|m\right|}{\left|y\right|^d}.\]

We now apply Corollary \ref{cor:bombieripoly} with $\eta=0.05$ and get
\[Ê\left|\frac{x}{y}-\xi\right| \geq \frac{1}{C(d,a)\left|y\right|^{\kappa}}\]
for $\left|a\right| \geq a_0(d,\eta)$, where
\[\kappa = \hat{\kappa}+\eta < 22.94+\eta=22.99\]
and $a_0(d,\eta)Ê\leq 2^{1196d}$ by \eqref{eq:aboundonazero}.

If $\left|a\right| \geq 2^{1196d} \geq a_0(d,\eta)$, we can combine the upper and the lower bound from above and conclude that
\[Ê\left|y\right|^{d-22.99} \leq \left|y\right|^{d-\kappa} \leq C(d,a)\left|m\right|.\]

Using \eqref{eq:nearroot} and $\left|\xi\right| < \left|a\right|+1 < 2\left|a\right|$, we get
\[Ê\max\{\left|x\right|,\left|y\right|\} \leqÊ\left|m\right|^{\frac{1}{d}}+2\left|a\right|\left(C(d,a)\left|m\right|\right)^{\frac{1}{d-22.99}}\]
and therefore
\[Ê\max\{\left|x\right|,\left|y\right|\} \leq 4\left|a\right|C(d,a)^{\frac{1}{d-22.99}}\left|m\right|^{\frac{1}{d-22.99}}.\]
As $\frac{d}{d-22.9} \geq 1$, it follows that
\[\max\{\left|x\right|,\left|y\right|\} \leq 2^{\frac{84d^2+2d}{d-22.99}}\left|a\right|^{\frac{28d^2+d}{d-22.99}}\left|m\right|^{\frac{1}{d-22.99}}.\]
Using $2^d \leq \left|a\right|^{\frac{1}{1196}}$ as well as $d \geq 23$, we get the desired upper bound.

If $\left|a\right| < 2^{1196d}$, we can use Theorem \ref{thm:Baker} with $\mathcal{H}=2^{1196d}$ and get the desired bound, using that $d \leq 2^d$. This completes the proof.
\end{proof}

\section{Applications II}\label{sec:ApplicationsII}
\begin{thm}
Let $d$ and $a$ be integers with $d \geq 25$, $d$ odd and $\left|a\right| \geq 2^{164d}$. Then the equation
\[Ê(x-ay)(x^2+y^2)^{\frac{d-1}{2}}-y^d=x+y\]
has at most $11$ solutions in integers $x$ and $y$.
\end{thm}

The equation has at least three solutions $(0,0)$ and $(\pm1,0)$ for any such $d$ and $a$.
If $a=-b^{d-1}-1$ for $b \in \mathbb{N}$, there are at least two further solutions $\pm(ab,b)$. It appears likely that these are all the solutions there are, at least for $\left|a\right|$ large enough, but a proof of this seems to be out of reach with our method.

\begin{proof}
We write
\[ÊA(x,y)=(x-ay)(x^2+y^2)^{\frac{d-1}{2}}-y^d.\]
We assume that $x$ and $y$ are integers, satisfying $A(x,y)=x+y$. We see that $-x$ and $-y$ satisfy the equation as well, since both its sides are homogeneous of odd degree. Therefore the solutions come in pairs $\pm(x,y)$ (except for $(0,0)$) and we can assume that $y \geq 0$. We will be mainly concerned with bounding the number of possibilities for $y$ and show in the end how the theorem follows from this. In the following we assume that $y > 0$.

Since $\left|a\right| \geq 5$, we can apply Lemma \ref{lem:rouche} with $S=\{1\}$ to
\[A(t)=(t-a)(t^2+1)^{\frac{d-1}{2}}-1\]
and write
\begin{equation}\label{eq:factors}
x+y=A(x,y)=y^dA\left(\frac{x}{y}\right)=\prod_{k=1}^{d}{(x-\xi_k y)},
\end{equation}
where $\left|\xi_k\right| < 2$ for $k < d$ and $\xi_{d}=\xi$ with $\left|\xi-a\right|<1$.

We choose $j$ such that
\begin{equation}\label{eq:thefifthelephant}
\left|x-\xi_j y\right| = \min_{k=1,\hdots,d}{\left|x-\xi_k y\right|}.
\end{equation}
Since $\left|x+y\right| \leq 2\max\{\left|x\right|,y\}$, it follows from \eqref{eq:factors} and the definition of $j$ that
\begin{equation}\label{eq:anapproximation}
Ê\left|x-\xi_j y\right| \leq 2^{\frac{1}{d}}\max\{\left|x\right|,y\}^{\frac{1}{d}}.
\end{equation}

From now on we assume that $y \geq 3$. We either have $\left|x\right| \leq y$ or $\left|x\right| > y$. In the latter case, it follows that $\left|x\right| > 3 \geq 2^{\frac{d+1}{d-1}}$ and therefore $2^{\frac{1}{d}}\left|x\right|^{\frac{1}{d}}Ê\leq \frac{\left|x\right|}{2}$. As $\left|x\right| > y$, this implies together with \eqref{eq:anapproximation} that
\[ \left|x\right|Ê\leq \left|\xi_j y\right|+\left|x-\xi_j y\right|Ê\leq \left|\xi_j y\right|+2^{\frac{1}{d}}\left|x\right|^{\frac{1}{d}} \leq \left|\xi_j y\right|+\frac{\left|x\right|}{2}.\]
We deduce that
\begin{equation}\label{eq:xintermsofy}
\left|x\right| \leq 2\max\{1,\left|\xi_j\right|\}y,
\end{equation}
which also holds if $\left|x\right| \leq y$.

If $j<d$ and $\Im \xi_j \geq 0$, we see that
\[Ê\left|\xi_j-i\right|=\left(\frac{1}{\left|\xi_j+i\right|^{\frac{d-1}{2}}\left|\xi_j-a\right|}\right)^{\frac{2}{d-1}} \leq \frac{1}{(\left|a\right|-2)^{\frac{2}{d-1}}} \leq \frac{1}{2}\]
and hence
\begin{equation}\label{eq:nightwatch}
\left|x-\xi_jy\right| \geq y\Im \xi_j \geq \frac{y}{2}.
\end{equation}
If $\Im \xi_j \leq 0$, the same inequality follows by interchanging $i$ and $-i$.

On the other hand it follows from $\left|\xi_j\right| < 2$, \eqref{eq:anapproximation} and \eqref{eq:xintermsofy} that
\[Ê\left|x-\xi_j y\right| \leq 2^{\frac{3}{d}}y^{\frac{1}{d}}.\]
Combining this with \eqref{eq:nightwatch}, we conclude that $y^{\frac{d-1}{d}} \leq 2^{\frac{d+3}{d}}$ and therefore
\[Êy \leq 2^{\frac{d+3}{d-1}} \leq 2^{\frac{28}{24}} \leq 3.\]
Thus we have proven that $y \leq 3$, if $y \geq 3$, and hence $y \leq 3$ unconditionally in the case $j < d$.

If $j=d$ and again $y \geq 3$, it follows from \eqref{eq:xintermsofy} and $\left|\xi\right|<\left|a\right|+1$ that $\left|x\right| \leq 2(\left|a\right|+1)y$. Together with \eqref{eq:factors}, this implies that
\begin{equation}\label{eq:snuff}
\prod_{k=1}^{d}{\left|\frac{x}{y}-\xi_k\right|} \leq \frac{\left|x\right|+y}{y^d} \leq \frac{2\left|a\right|+3}{y^{d-1}}.
\end{equation}
We see that
\[Ê\left|\frac{x}{y}-\xi_k\right| \geqÊ\left|\xi-\xi_k\right|-\left|\frac{x}{y}-\xi\right| \geq \left|\xi-\xi_k\right|-\left|\frac{x}{y}-\xi_k\right|\]
for $k<d$ because of \eqref{eq:thefifthelephant} and therefore
\[\left|\frac{x}{y}-\xi_k\right| \geqÊ\frac{\left|\xi-\xi_k\right|}{2} \geq \frac{\left|a\right|-3}{2}.\]
Thus we deduce from \eqref{eq:snuff} that
\begin{equation}\label{eq:evenbetterapproximation}
\left|\frac{x}{y}-\xi\right| \leq \frac{2^{d-1}(2\left|a\right|+3)}{(\left|a\right|-3)^{d-1}y^{d-1}} \leq \frac{\frac{9}{2}\cdot2^{d-2}}{(\left|a\right|-3)^{d-2}y^{d-1}} \leq \frac{\frac{9}{2}\cdot4^{d-2}}{\left|a\right|^{d-2}y^{d-1}},
\end{equation}
since $\frac{2\left|a\right|+3}{\left|a\right|-3}=2+\frac{9}{\left|a\right|-3} \leq \frac{9}{4}$ and $\frac{\left|a\right|-3}{2} \geqÊ\frac{\left|a\right|}{4}$.

Now suppose that we have two such solutions $(x,y)$ and $(x',y')$. If $\frac{x}{y}=\frac{x'}{y'}$, it follows that $(x',y')=\lambda(x,y)$ for some rational $\lambdaÊ\neq 0$, and hence
\[\lambda^d(x+y)=\lambda^dA(x,y)=A(x',y')=x'+y'=\lambda(x+y),\]
where $A(x',y')=y'^{d}A\left(\frac{x'}{y'}\right)$. If $x+yÊ= 0$, it follows that $A(x,y)=0$ and therefore also $A\left(\frac{x}{y}\right)=0$. But this contradicts the irreducibility of $A(t)$ over $\mathbb{Q}$ provided by Lemma \ref{lem:rouche}(b), so we deduce that $x+y \neq 0$. It follows that $\lambda^{d-1}=1$, which has exactly two real solutions $\lambda = \pm 1$, so $(x',y')=\pm(x,y)$. Since $y>0$ and $y'>0$, this is only possible if $(x,y)=(x',y')$. It follows that $\frac{x}{y} \neq \frac{x'}{y'}$, if $(x,y)Ê\neq (x',y')$. If further $y \leq y'$, it follows from \eqref{eq:evenbetterapproximation} that
\[Ê\frac{1}{yy'} \leq \left|\frac{x}{y}-\frac{x'}{y'}\right| \leq \left|\frac{x}{y}-\xi\right|+\left|\frac{x'}{y'}-\xi\right| \leq \frac{9\cdot4^{d-2}}{\left|a\right|^{d-2}y^{d-1}}\]
and hence
\[Êy' \geq \frac{1}{9}\left(\frac{\left|a\right|}{4}\right)^{d-2}y^{d-2}.\]

If there are $n$ different such solutions $(x_k,y_k)$ ($k=1,\hdots,n$) with
\[4 \leq y_1 \leq \hdots \leq y_n,\]
and $n \geq 3$, it follows that
\[Êy_2 \geq \frac{1}{9}\left(\frac{\left|a\right|}{4}\right)^{d-2}y_1^{d-2} \geq \frac{1}{9}4^{d-2}\left(\frac{\left|a\right|}{4}\right)^{d-2} \geq 4^{d-4}\left(\frac{\left|a\right|}{4}\right)^{d-2},Ê\]
\[Êy_3 \geq \frac{1}{9}\left(\frac{\left|a\right|}{4}\right)^{d-2}y_2^{d-2} \geq \frac{1}{9}4^{(d-2)(d-4)}\left(\frac{\left|a\right|}{4}\right)^{(d-1)(d-2)} \geq 4^{(d-4)^2}\left(\frac{\left|a\right|}{4}\right)^{(d-2)^2}\]
and then by induction that
\begin{equation}\label{eq:exponentiallowerboundtwo}
y_n \geq 4^{(d-4)^{n-1}}\left(\frac{\left|a\right|}{4}\right)^{(d-2)^{n-1}}.
\end{equation}

On the other hand, Corollary \ref{cor:combieripoly} with $\eta=0.56$ implies that
\[Ê\left|\frac{x_n}{y_n}-\xi\right| \geq \frac{1}{C(d,a)y_n^{23.5}},\]
since $\kappa=\hat{\kappa}+\eta<22.94+\eta=23.5$ and $a_0(d,0.56) \leq 2^{\frac{3.9 \cdot 23.5d}{0.56}}+1 \leq 2^{164d} \leq \left|a\right|$. Together with \eqref{eq:evenbetterapproximation}, this implies that
\[y_n^{0.5}Ê\leq y_n^{d-24.5}Ê\leq \frac{9}{2}C(d,a)\frac{4^{d-2}}{\left|a\right|^{d-2}} \leq \frac{9}{2}C(d,a).\]
Using $d \geq 25$, we deduce that
\begin{equation}\label{eq:thanksforallthefish}
y_n \leq \left(\frac{9}{2}C(d,a)\right)^2 \leq 2^{189d^2}\left|a\right|^{56d^2}.
\end{equation}

If $n \geq 5$, it follows from \eqref{eq:exponentiallowerboundtwo} that
\[Êy_n \geq 4^{(d-4)^{n-1}}\left(\frac{\left|a\right|}{4}\right)^{(d-2)^{n-1}} \geq 4^{(d-4)^4}\left(\frac{\left|a\right|}{4}\right)^{(d-2)^4} \geq 2^{2(d-4)^4}\left|a\right|^{\frac{(d-2)^{4}}{2}},\]
because $\left|a\right| \geq 16$ and hence $\frac{\left|a\right|}{4} \geq \left|a\right|^{\frac{1}{2}}$. Furthermore, it follows from $d \geq 25$ that
\[Ê2^{2(d-4)^4}\left|a\right|^{\frac{(d-2)^4}{2}} \geq 2^{882(d-4)^2}\left|a\right|^{\frac{529}{2}(d-2)^2} > 2^{\frac{14112}{25}d^2}\left|a\right|^{\frac{4232}{25}d^2},\]
as $d-2 > d-4 > \frac{4}{5}d$. We deduce that
\[Êy_n > 2^{\frac{14112}{25}d^2}\left|a\right|^{\frac{4232}{25}d^2} > 2^{564d^2}\left|a\right|^{169d^2}.\]
But this contradicts \eqref{eq:thanksforallthefish}, so we conclude that $n \leq 4$.

Summarizing, we have $y \leq 3$ for all but $n \leq 4$ solutions $(x,y)$. Since the solutions (apart from $(0,0)$) come in pairs $\pm(x,y)$, it follows that $\left|y\right| \leq 3$ for all but $2nÊ\leq 8$ solutions $(x,y)$. If $y=0$, we see that there are exactly three solutions $x=0$, $\pm1$.

If $0 < \left|y\right| \leq 3$, we define $Q(t)=(t^2+y^2)^{\frac{d-1}{2}}$ and $P(t)=-t-y^d-y$ and put $\tilde{a}=ay$. We see that $P$ and $Q$ are coprime and $Q$ is monic. Furthermore $P$ has degree $d_0=1$, $L(P) \leq 3^{d+1}$ and
\[R=\max_{Q(\theta)=0}{\left|\theta\right|}=\left|y\right| \leq 3.\]
It follows that
\[ \left|\tilde{a}\right| \geq \left|a\right| \geq 2^{164d} \geq 4\cdot3^{d+1}+8 \geq L(P)\max\left\{2^{\frac{d_0}{d-d_0-1}},(R+1)^{d_0}\right\}+2R+2,\]
so the polynomial $\tilde{A}(t)=(t-\tilde{a})Q(t)+P(t)$ is irreducible over $\mathbb{Q}$ by Lemma \ref{lem:rouche} (note that $L(P) \geq 1$).
As $\tilde{A}(x)=A(x,y)-x-y$, this implies in particular that there is no integer $x$ with $A(x,y)=x+y$ and now the theorem follows.
\end{proof}

\section*{Acknowledgements}
I thank David Masser for his innumerable comments on the work (originally my master thesis) and his help in preparing it for publication.

\bibliographystyle{abbrv}
\bibliography{Bibliography}

\vfill

\end{document}